\documentclass[reqno,tbtags,a4paper,12pt]{amsart}

\usepackage[in]{fullpage}
\usepackage[matrix,arrow,curve]{xy}
\usepackage{tikz,tikz-cd}
\usetikzlibrary{decorations.pathreplacing}
\usetikzlibrary{decorations.markings}

\usepackage{amsmath,amssymb,amsthm,amsfonts,amscd}
\usepackage{graphics,graphicx}

\usepackage{mathrsfs}
\renewcommand{\mathscr}{\mathcal}

\newcommand{\Mod}{\mathop{\mathrm{Mod}}\nolimits}

\newcommand{\im}{\mathop{\rm im}\nolimits}

\newcommand{\Sp}{\mathrm{Sp}}
\newcommand{\SL}{\mathrm{SL}}
\newcommand{\GL}{\mathrm{GL}}
\newcommand{\IA}{\mathrm{IA}}
\newcommand{\IO}{\mathrm{IO}}
\newcommand{\Aut}{\mathop{\mathrm{Aut}}\nolimits}
\newcommand{\Out}{\mathop{\mathrm{Out}}\nolimits}
\newcommand{\Inn}{\mathop{\mathrm{Inn}}\nolimits}

\newcommand{\ab}{\mathrm{ab}}
\newcommand{\Hom}{\mathop{\mathrm{Hom}}\nolimits}

\newcommand{\I}{\mathcal{I}}
\newcommand{\K}{\mathcal{K}}
\newcommand{\Stab}{\mathop{\mathrm{Stab}}\nolimits}

\newcommand{\Z}{\mathbb{Z}}
\newcommand{\Q}{\mathbb{Q}}
\newcommand{\bC}{\mathbb{C}}
\newcommand{\bN}{\mathbb{N}}

\newcommand{\be}{\mathbf{e}}
\newcommand{\bk}{\mathbf{k}}

\newcommand{\bm}{\mathbf{m}}
\newcommand{\bn}{\mathbf{n}}
\newcommand{\bq}{\mathbf{q}}
\newcommand{\bnul}{\mathbf{0}}
\newcommand{\bz}{\mathbf{z}}
\newcommand{\bfy}{\mathbf{y}}

\newcommand{\CG}{\mathcal{G}}

\newcommand{\CR}{\mathcal{R}}

\newcommand{\CQ}{\mathcal{Q}}
\newcommand{\CT}{\mathcal{T}}

\newcommand{\tkappa}{\widetilde{\kappa}}
\newcommand{\tlambda}{\widetilde{\lambda}}
\newcommand{\tmu}{\widetilde{\mu}}

\newcommand{\rk}{\mathop{\mathrm{rank}}\nolimits}
\newcommand{\exter}{\wedge}
\newcommand{\oS}{\overline{S}}

\newcommand{\tV}{\widetilde{V}}

\newcommand{\fX}{\mathfrak{X}}

\newtheorem{thm}{Theorem}
\newtheorem*{addthm}{Addendum to Theorem~\ref{thm_nilpotent}}

\newtheorem{theorem}{Theorem}[section]
\newtheorem{propos}[theorem] {Proposition}
\newtheorem{cor}[theorem] {Corollary}
\newtheorem{lem}[theorem]{Lemma}
\theoremstyle{definition}

\newtheorem{defin}[theorem]{Definition}
\newtheorem{remark}[theorem]{Remark}
\newtheorem{quest}[theorem]{Question}
\newtheorem{problem}[theorem]{Problem}

\numberwithin{equation}{section}

\author{Alexander A. Gaifullin}

\thanks{This work was performed at the Steklov International Mathematical Center and supported by the Ministry of Science and Higher Education of the Russian Federation (agreement no. 075-15-2025-303) and Theoretical Physics and Mathematics Advancement Foundation \mbox{``BASIS''} (grant 22-7-2-10-1).}

\title[Finite generation of abelianizations]{Finite generation of abelianizations of the genus~3 Johnson kernel  and the commutator subgroup of the Torelli group for~$\Out(F_3)$}
\date{}

\address{Steklov Mathematical Institute of Russian Academy of Sciences, Moscow, Russia}
\address{Lomonosov Moscow State University, Russia}


\begin{document}

\begin{abstract}
 Let $\Sigma_g^b$ be a compact oriented surface of genus~$g$ with $b$ boundary components, where $b\in\{0,1\}$. The \textit{Johnson kernel}~$\K_g^b$ is the subgroup of the mapping class group~$\Mod(\Sigma_g^b)$ generated by Dehn twists about separating simple closed curves. Let $F_n$ be a free group with $n$ generators. The \textit{Torelli group} for~$\Out(F_n)$ is the subgroup~$\IO_n\subset\Out(F_n)$ consisting of all outer automorphisms that act trivially on the abelianization of~$F_n$. Long standing questions are whether the groups~$\K_g^b$ and~$[\IO_n,\IO_n]$ or their abelianizations~$(\K_g^b)^{\ab}$ and~$[\IO_n,\IO_n]^{\ab}$ are finitely generated for $g\ge3$ (respectively, $n\ge3$). During the last $15$ years, these questions were answered positively for $g\ge4$ and~$n\ge4$, respectively. Nevertheless, the cases of~$g=3$ and~$n=3$ remained completely unsettled. In this paper, we prove that the abelianizations~$(\K_3^b)^{\ab}$ and $[\IO_3,\IO_3]^{\ab}$ are finitely generated. Our approach is  based on a new general sufficient condition for a module over a Laurent polynomial ring to be finitely generated as an abelian group.
\end{abstract}

\maketitle

\section{Introduction}

We denote by~$G^{\ab}$ the abelianization of a group~$G$ and always use additive notation for~$G^{\ab}$. For an element $x\in G$, we denote by~$[x]$ its image in~$G^{\ab}$. Our convention for the group commutator is
$$
(x,y)=xyx^{-1}y^{-1}.
$$

\subsection{Johnson kernel}
Consider an oriented surface~$\Sigma_g^b$ of genus~$g$ with  $b$ boundary components, where $b\in\{0,1\}$, and let $H=H_1(\Sigma_g^b,\Z)$. (If $b=0$, we will usually omit the superscript~$b$ in all notations.) The \textit{mapping class group}~$\Mod(\Sigma_g^b)$ is the group consisting of the isotopy classes of orientation-preserving homeomorphisms of~$\Sigma_g^b$. If $b\ne 0$, then the homeomorphisms are assumed to fix the boundary pointwise. The \textit{Torelli subgroup}~$\I_g^b$ is defined to be the kernel of the action of~$\Mod(\Sigma_g^b)$ on~$H$:
\begin{equation*}
 1\longrightarrow \I_g^b\longrightarrow\Mod(\Sigma_g^b)\longrightarrow \Sp(H)\longrightarrow 1.
\end{equation*}
Here $\Sp(H)\cong\Sp_{2g}(\Z)$ is the group of all automorphisms of~$H$ that preserve the intersection form. The importance of Torelli groups stems from their fundamental connections to algebraic geometry, arithmetic geometry, three-dimensional topology, and automorphisms of free groups. For more details on Torelli groups, their connections, and associated open problems, see the survey papers~\cite{Mor99}, \cite{Far06}, \cite{Hai06}, \cite{Mar19}, \cite{Hai20}.

Starting from a remarkable series of papers~\cite{Joh83}--\cite{Joh85b} by Johnson, a key role in the study of the Torelli group~$\I_g^b$ is played by its subgroup~$\K_g^b$ generated by Dehn twists about separating simple closed curves. This subgroup is usually called the \textit{Johnson kernel}, since Johnson~\cite{Joh85a} proved that it coincides with the kernel of the \textit{Johnson homomorphism}~$\tau$, which arises from the action of~$\I_g^b$ on the nilpotent group $\pi/[\pi,[\pi,\pi]]$, where $\pi$ is the fundamental group of~$\Sigma_g^b$, see~\cite{Joh80b}. Namely, we have
\begin{equation*}
 1\longrightarrow \K_g^b\longrightarrow \I_g^b\stackrel{\tau}{\longrightarrow} U_g^b\longrightarrow 1,
\end{equation*}
where
\begin{align*}
U_g^1&=\exter^3H\cong\Z^{\binom{2g}3},\\
U_g^0&=(\exter^3H)/(\Omega\wedge H)\cong\Z^{\binom{2g}3-2g}.
\end{align*}
Here $\Omega\in\exter^2H$ is the dual of the intersection form. The significance of the groups~$\K_g^b$ is connected with Morita's seminal work~\cite{Mor89}, which revealed their fundamental role in the theory of the Casson invariant for three-dimensional homology spheres.

By a classical result of Johnson~\cite{Joh83} the groups~$\I_g^b$ are finitely generated, provided that $g\ge 3$.
The groups $\K_2^b=\I_2^b$ are known to be not finitely generated (McCullough and Miller~\cite{MCM86}) and free (Mess~\cite{Mes92}). A long-standing question, see~\cite{MCM86}, \cite[Question~10]{Mor94}, \cite[Problem~2.2]{Mor99}, is whether the groups~$\K_g^b$ are finitely generated for $g\ge 3$ and whether their abelianizations $(\K_g^b)^{\ab}=H_1(\K_g^b,\Z)$ are finitely generated for $g\ge 3$. For many years, it was expected that the answers to both questions would be negative. However, in the groundbreaking work~\cite{DiPa13} by Dimca and Papadima, it was shown that $H_1(\K_g,\Q)$ is finite-dimensional for $g\ge 4$. Later Ershov and He~\cite{ErHe18} esteblished that $\K_g^b$ is finitely generated for $g\ge 12$, and then Church, Ershov, and Putman~\cite{CEP22} extended this result to all $g\ge 4$.

The case $g=3$ has remained completely unsettled. The first main result of the present paper is as follows.

\begin{thm}\label{thm_main_K}
Suppose that $g\ge 3$ and~$b\in\{0,1\}$. Then $(\K_g^b)^{\ab}$ is a finitely generated abelian group.
\end{thm}

\begin{remark}
 We formulate this theorem for all $g\ge 3$, since we are going  to give a proof of it that works uniformly for all $g\ge 3$. Certainly, for $g\ge 4$, this result is not new, as it follows directly from the finite generation of~$\K_g^b$. Nevertheless, our proof of the finite generation of~$(\K_g^b)^{\ab}$ for $g\ge 4$ seems to be easier, and definitely less involved, than all previously known proofs of this fact. The same applies to Theorem~\ref{thm_main_OA} below.
\end{remark}

\begin{remark}\label{remark_filtration}
 The groups $\I_g^b(1)=\I_g^b$ and $\I_g^b(2)=\K_g^b$ are the first two terms of the \textit{Johnson filtration} $$\I_g^b(1)\supset \I_g^b(2)\supset \I_g^b(3)\supset\cdots$$ Here the group $\I_g^b(k)$ is defined to be the kernel of the natural action of~$\I_g^b$ on $\pi/\gamma_{k+1}\pi$, where $\pi=\pi_1(\Sigma_g^b)$ and $\{\gamma_k\pi\}_{k\in\bN}$ is the lower central series of~$\pi$. In~\cite{ErHe18} and~\cite{CEP22} results of finite generation were obtained not only for $\K_g^b$ with $g\ge 4$ but also for some higher terms of the Johnson filtration, and hence for their abelianizations. The author does not know whether the methods of the present paper can be applied to improve these results.
\end{remark}

\subsection{Automorphisms of free groups}

Let $F_n$ be a free group on $n$ generators. A deep analogy exists between the automorphism groups~$\Aut(F_n)$ and the outer automorphism groups~$\Out(F_n)=\Aut(F_n)/\Inn(F_n)$ on one hand, and the mapping class groups~$\Mod(\Sigma_g^b)$ on the other. This analogy has proved remarkably fruitful for understanding both types of groups. The \textit{Torelli subgroup}~$\IA_n$ (repectively, $\IO_n$) is defined to be the kernel of the action of~$\Aut(F_n)$ (respectively, $\Out(F_n)$) on the abelianization $F_n^{\ab}=\Z^n$:
\begin{equation*}
\begin{tikzcd}
 1\arrow[r]
 & \IA_n \arrow[d, two heads] \arrow [r] & \Aut(F_n) \arrow[d, two heads] \arrow[r]
 & \GL_n(\Z)\arrow[r] \arrow [d, equal] & 1\\
 1\arrow[r]
 & \IO_n \arrow [r] & \Out(F_n)  \arrow[r]
 & \GL_n(\Z)\arrow[r] & 1
\end{tikzcd}
\end{equation*}
A classical result of Nielsen~\cite{Nie18} states that $\IA_2=\Inn(F_2)\cong F_2$ and hence the group~$\IO_2$ is trivial. Another classical result due to Nielsen~\cite{Nie24} for $n=3$ and to Magnus~\cite{Mag35} for $n>3$ (cf.~\cite{MKS}) says that the groups~$\IA_n$, and hence the groups~$\IO_n$, are finitely generated. The following commutative diagram with exact rows was obtained by Formanek~\cite{For90}:
\begin{equation}\label{eq_exact_IA}
 \begin{tikzcd}
 1\arrow[r]
 & \left[\IA_n,\IA_n\right] \arrow[d, two heads] \arrow [r] & \IA_n \arrow[d, two heads] \arrow[r,"\tau"]
 & H^*\otimes\exter^2H \arrow [d, two heads] \arrow[r]  & 1\\
 1\arrow[r]
 & \left[\IO_n,\IO_n\right] \arrow [r] & \IO_n  \arrow[r,"\tau"]
 & (H^*\otimes\exter^2H)/\iota(H) \arrow[r] & 1
\end{tikzcd}
\end{equation}
Here $H=F_n^{\ab}=\Z^n$, $H^*=\Hom(H,\Z)$ and the embedding $\iota\colon H\hookrightarrow H^*\otimes\exter^2H$ is induced by the canonical embedding $\Z\hookrightarrow H^*\otimes H=\Hom(H,H)$ that takes~$1$ to the identity operator. By analogy with the mapping  class group case, the surjective homomorphisms~$\tau$ in~\eqref{eq_exact_IA} are called \textit{Johnson's homomorphisms}. (We conveniently denote all kinds of Johnson's homomorphisms by~$\tau$; this does not lead to a confusion.)

Natural analogues of the questions about~$\K_g^b$ discussed above are the questions of whether the group~$[\IO_n,\IO_n]$ and its abelianization~$[\IO_n,\IO_n]^{\ab}$ are finitely generated for $n\ge3$, and similar questions for~$[\IA_n,\IA_n]$. Papadima and Suciu~\cite{PaSu12} proved that $H_1([\IA_n,\IA_n],\Q)$  for $n\ge 5$ and $H_1([\IO_n,\IO_n],\Q)$ for $n\ge 4$ are finite-dimensional. Then Ershov and He~\cite{ErHe18} showed that the groups~$[\IA_n,\IA_n]$, and hence the groups~$[\IO_n,\IO_n]$, are finitely generated for all $n\ge 4$. The natural expectation was that the group~$[\IA_3,\IA_3]$ would be not finitely generated, and moreover, would have an infinite-dimensional first rational homology. (See, e.\,g.,~\cite[Section~9.3]{ErHe18}, which provides several arguments  towards this, and also Remark~\ref{remark_Chein} below.) However, our second main result shows that at least for the question concerning the abelianization of~$[\IO_3,\IO_3]$ the answer is positive.

\begin{thm}\label{thm_main_OA}
If $n\ge 3$, then $[\IO_n,\IO_n]^{\ab}$ is a finitely generated abelian group.
\end{thm}

\begin{remark}
 Though our approach does not work for $[\IA_3,\IA_3]^{\ab}$, it still allows to prove that  $[\IA_n,\IA_n]^{\ab}$ is finitely generated for $n\ge 4$, which is, however, not a new result; see Remark~\ref{remark_IA} for more detail.
\end{remark}

\begin{remark}
 For automorphism groups, there exist analogues
 \begin{align*}
 \IA_n(1)&\supset\IA_n (2)\supset \IA_n(3)\supset \cdots\\
 \IO_n(1)&\supset\IO_n (2)\supset \IO_n(3)\supset \cdots
 \end{align*}
 of the Johnson filtration (see Remark~\ref{remark_filtration}), which in this case were introduced by Andreadakis~\cite{And65} even before Johnson's papers~\cite{Joh85a}, \cite{Joh85b} and are usually referred to as the \textit{Andreadakis filtrations} (or \textit{Andreadakis--Johnson filtrations}). These filtrations start from the groups~$\IA_n(1)=\IA_n$ and $\IO_n(1)=\IO_n$. It is also well known that $\IA_n(2)=[\IA_n,\IA_n]$ (see~\cite{And65} for $n\le 3$ and~\cite[Lemma~5]{Bac66} for any~$n$) and hence $\IO_n(2)=[\IO_n,\IO_n]$. Some results about the finite generation of higher terms of the Andreadakis filtrations were also obtained in~\cite{ErHe18} and~\cite{CEP22}. As in the mapping class group case, the author does not know whether the methods of the present paper can be applied to improve these results.
\end{remark}

\subsection{Method and a general result}
The approach employed in~\cite{DiPa13} and~\cite{PaSu12}, as well as in subsequent papers~\cite{DHP14}, \cite{ErHe18}, \cite{CEP22}, was based on the systematic use of the arithmetic symmetry group~$\Sp_{2g}(\Z)$ (respectively, $\SL_n(\Z)$) and on the study of certain invariants of the Torelli group~$\I_g$ (respectively, $\IA_n$ or~$\IO_n$) possessing this symmetry and related to the finiteness properties of~$\K_g$ (respectively, $[\IA_n,\IA_n]$ or~$[\IO_n,\IO_n]$). Among such invariants were the resonance varieties~$\CR^i_k$, the characteristic varieties ~$\mathcal{V}^i_k$, and the Bieri--Neumann--Strebel--Renz invariants~$\Sigma^q$.
The case $g=3$ is exceptional from an arithmetic perspective. This is primarily due to the fact that the resonance variety $\CR_1^1(\I_g)$ consists solely of the point~$0$ for $g\ge 4$, while coinciding with the entire cohomology group~$H^1(\I_3,\bC)$ for $g=3$, see~\cite[Theorem~4.4]{DiPa13}. (Recall that, by definition, $\CR_1^1(\I_g)$ is the subset of~$H^1(\I_g,\bC)$ consisting of all classes~$a$ such that $H^1\bigl(H^{\bullet}(\I_g,\bC),\mu_a\bigr)\ne 0$, where $\mu_a$ denotes the left multiplication by~$a$.) For the groups~$\IA_n$, the case $n=3$ is also exceptional. For instance, there exists a nontrivial one-dimensional complex representation~$V$ of~$\IA_3^{\ab}$ such that $H^1(\IA_3,V)\ne 0$, see~\cite[Proposition~9.5]{ErHe18}, while for $n\ge 4$, $H^1(\IA_n,V)=0$ for any nontrivial irreducible representation~$V$ of~$\IA_n^{\ab}$ over an arbitrary field, see~\cite[Theorem~1.4]{ErHe18}.

Our approach is different. We also use the symmetry group~$\Sp_{2g}(\Z)$ (or~$\SL_n(\Z)$), but its arithmetic nature is less important for us. The only thing we require from this group is a rather weak property, which we call the \textit{subgroup displacement property} (see Definition~\ref{defin_SDP} below). In fact, we will prove a very general sufficient condition for a finitely generated module~$M$ over the group ring~$\Z A$ (where $A$ is a finitely generated free abelian group) to itself be finitely generated as an abelian group, see Theorem~\ref{thm_main_gen} below. From this sufficient condition, Theorem~\ref{thm_main_K} will follow almost immediately, while Theorem~\ref{thm_main_OA} will be deduced relatively easily. The key aspect of our approach is that we focus on the \textit{$S_{\fX}$-torsion condition}  with respect to a very special multiplicatively closed subset~$S_{\fX}$ of~$\Z A$, which we call a \textit{difference multiplicatively closed subset}. Both the definiton of the subgroup displacement property and the construction of the difference multiplicatively closed subset~$S_{\fX}$ will depend on the choice of a subset $\fX\subset A\setminus\{0\}$. The main option for~$\fX$, which the reader should keep in mind, is the set of all nonzero elements of~$A$. This choice will be sufficient in the cases of the groups~$\K_g$ and~$[\IO_n,\IO_n]$. In the case of~$\K_g^1$, a slightly more intricate set~$\fX$ will be required.

\begin{remark}
 We will always use additive notation for abelian groups. To avoid confusion, we will denote the basis element of the group ring~$\Z A$ corresponding to an element $a\in A$ by~$e_a$; then $e_ae_b=e_{a+b}$. On the other hand, for a non-abelian group in multiplicative notation, we will usually write~$g$ instead of~$e_g$.
\end{remark}

Let $A$ be a free abelian group of finite rank~$r$ and $\fX$ a nonempty subset of~$A\setminus\{0\}$.

\begin{defin}\label{defin_SDP}
 Suppose that a group~$D$ acts on~$A$ via a homomorphism $D\to\GL(A)$. We say that this action possesses the \textit{subgroup displacement property}  with respect to~$\fX$ if, for each subgroup $B\subset A$ with $\rk(B)<\rk(A)$ (i.\,e., with $|A:B|=\infty$) and each finite subset $X\subset \fX$, there exists an element $g\in D$ such that the subgroup~$g(B)$ is disjoint from~$X$.
\end{defin}

\begin{defin}
The \textit{difference multiplicatively closed subset}~$S_{\fX}$ is the multiplicatively closed subset of~$\Z A$ generated by all elements $e_a-1$, where $a\in \fX$. In other words, $S_{\fX}$ consists of all elements of the form
$$
(e_{a_1}-1)\cdots (e_{a_k}-1),
$$
where $k\ge 0$ and $a_1,\ldots,a_k\in \fX$.
\end{defin}
Recall that an element $x$ of a $\Z A$-module~$M$ is called an \textit{$S_{\fX}$-torsion element} if there exists an element $s\in S_{\fX}$ such that $sx=0$ (equivalently, if $x$ goes to zero after the localization with respect to~$S_{\fX}$). A module is called an \textit{$S_{\fX}$-torsion module} if it consists of $S_{\fX}$-torsion elements.

Our general sufficient condition for a $\Z A$-module to be a finitely generated abelian group is as follows.

\begin{thm}\label{thm_main_gen}
 Let
 \begin{equation}\label{eq_ses}
 1\longrightarrow A\longrightarrow Q\longrightarrow D\longrightarrow 1
 \end{equation}
 be a short exact sequence of groups, where $A$ is a finitely generated free abelian group, and let $\fX$ be a nonempty subset of~$A\setminus\{0\}$. Assume that the action of~$D$ on~$A$ associated with the short exact sequence~\eqref{eq_ses} possesses the subgroup displacement property with respect to~$\fX$. Suppose that $M$ is a $\Z Q$-module such that, being considered as a $\Z A$-module, it is a finitely generated $S_{\fX}$-torsion module. Then $M$ is finitely generated as an abelian group.
\end{thm}

In fact, by eliminating the roles of the group~$D$ and the subset~$\fX$, we can give an even more general sufficient condition for a $\Z A$-module to be finitely generated as an abelian group. However, the condition imposed on~$M$ becomes more complicated. Though we need this result for only free abelian groups, we will conveniently formulate and prove it for an arbitrary finitely generated abelian group.

\begin{thm}\label{thm_more_gen}
 Suppose that $A$ is a finitely generated abelian group and $M$ is a finitely generated $\Z A$-module. Assume that the following condition holds:
 \begin{itemize}
  \item[$(*)$] For each subgroup $B\subset A$ with $|A:B|=\infty$, there exists a positive integer~$k$ and elements $a_1,\ldots,a_k\in A\setminus B$ such that the product
  \begin{equation*}
   (e_{a_1} -1)\cdots (e_{a_k} -1)
  \end{equation*}
  annihilates~$M$.
 \end{itemize}
 Then $M$ is finitely generated as an abelian group.
\end{thm}

\begin{remark}
 If $A$ is free abelian of rank~$r$, then $\Z A\cong \Z\bigl[t_1^{\pm1},\ldots,t_r^{\pm1}\bigr]$ is a Laurent polynomial ring and the elements~$e_a$ are monomials $t_1^{n_1}\cdots t_r^{n_r}$, so Theorem~\ref{thm_more_gen} provides a sufficient condition for a module over a Laurent polynomial ring to be finitely generated as an abelian group.
\end{remark}

Theorem~\ref{thm_main_gen} is a direct consequence of Theorem~\ref{thm_more_gen}. Indeed, suppose that $M$ is a module that satisfies the conditions of Theorem~\ref{thm_main_gen}. Since $M$ is a finitely generated $S_{\fX}$-torsion $\Z A$-module, we see that $M$ is annihilated by an element of the form
\begin{equation*}
   (e_{a_1} -1)\cdots (e_{a_k} -1),
\end{equation*}
where $a_1,\ldots,a_k\in\fX$. Then, for each~$g\in D$, the element
\begin{equation*}
   (e_{g(a_1)} -1)\cdots (e_{g(a_k)} -1)
\end{equation*}
also annihilates~$M$. Condition~$(*)$ now follows from the subgroup displacement property.

We will apply Theorem~\ref{thm_main_gen} in the following situation.

\begin{cor}\label{cor_main_gen}
 Suppose that $K\triangleleft T\triangleleft \Gamma$ is a chain of normal subgroups such that $K$ is normal in~$\Gamma$, $T$ is finitely generated, and the group $A=T/K$ is free abelian. Assume that there is a subset $\fX\subset A\setminus\{0\}$ such that
 \begin{enumerate}
  \item the induced action of the group~$D=\Gamma/T$ on~$A$ has the subgroup displacement property with respect to~$\fX$,
  \item $K^{\ab}$ is an $S_{\fX}$-torsion $\Z A$-module (where the structure of a $\Z A$-module is provided by the action of~$T$ on~$K$ by conjugations).
 \end{enumerate}
 Then $K^{\ab}$ is a finitely generated abelian group.
\end{cor}

\begin{remark}
 The only thing we need to check to deduce Corollary~\ref{cor_main_gen} from Theorem~\ref{thm_main_gen} is that $K^{\ab}$ is finitely generated as a $\Z A$-module (equivalently, as a $\Z T$-module). This fact for a normal subgroup $K$ of a finitely generated group~$T$ with an abelian, or even nilpotent, quotient~$T/K$ is well known, see, for instance, \cite[Theorem~8.3(a)]{ErHe18}.
\end{remark}

To deduce Theorems~\ref{thm_main_K} and~\ref{thm_main_OA} from Corollary~\ref{cor_main_gen} we need to check conditions~(1) and~(2) for the chains of subgroups
\begin{align*}
 \K_g^b\triangleleft\I_g^b\triangleleft\Mod(\Sigma_g^b)&,& g&\ge 3,\ b\in\{0,1\},\\
 [\IO_n,\IO_n]\triangleleft\IO_n\triangleleft\Out^+(F_n)&,& n&\ge 3,
\end{align*}
respectively. Here $\Out^+(F_n)$ is the index two subgroup of~$\Out(F_n)$ consisting of all automorphisms that act on~$F_n^{\ab}$ by an element of~$\SL_n(\Z)$.
In the cases of~$\K_g$ and~$[\IO_n,\IO_n]$ the subgroup displacement property holds for $\fX=A\setminus\{0\}$ and follows rather easily from the fact that in these cases the action of $D=\Sp_{2g}(\Z)$ (respectively, $D=\SL_n(\Z)$) on~$A$ extends to an irreducible rational representation of~$\Sp_{2g}(\bC)$ (respectively, $\SL_n(\bC)$) in~$A\otimes\bC$, see Section~\ref{section_SDP}. In the case of~$\K_g^1$, the $\Sp_{2g}(\bC)$-representation $A\otimes\bC$ is not irreducible, so the subgroup displacement property with respect to~$A\setminus\{0\}$ does not hold. Nevertheless, in Section~\ref{section_SDP} we will prove the subgroup displacement property with respect to an appropriate subset $\fX\subset A\setminus\{0\}$.

Next, we need to check the $S_{\fX}$-torsion condition. This condition for the groups~$(\K_g^b)^{\ab}$ and~$[\IO_n,\IO_n]^{\ab}$ will be provided by Propositions~\ref{propos_K_torsion} and~\ref{propos_IA_torsion}, respectively. The case of the groups~$\K_g^b$ is easy: it will  immediately follow from the fact that every twist about a separating simple closed curve commutes with a genus $1$ BP map. The case of~$[\IO_n,\IO_n]^{\ab}$ will require more substantial work.

\begin{remark}
 Our proof seems not to generalize to the group~$[\IA_3,\IA_3]^{\ab}$. Here, as for~$\K_g^1$, the $\SL_3(\bC)$-representation~$A\otimes \bC$ is reducible. Nevertheless, the subgroup displacement property can still be esteblished for an appropriately chosen subset~$\fX_1$. Moreover, the $S_{\fX_2}$-torsion property of~$[\IA_3,\IA_3]^{\ab}$ can also be proved for some reasonable subset~$\fX_2$. Unfortunately, these subsets~$\fX_1$ and~$\fX_2$ are distinct. The author has not managed to find any single subset~$\fX$ for which both the subgroup displacement property and the $S_{\fX}$-torsion property of~$[\IA_3,\IA_3]^{\ab}$ can be proved simultaneously.
\end{remark}

\subsection{Nilpotency property of~$\K_g^{\ab}$}\label{subs_nilpotence}

Let $G$ be a group, $R$ a commutative ring, and $I_G$ the augmentation ideal in~$RG$. A (left) $RG$-module~$M$ is called \textit{nilpotent} if $I_G^q\cdot M=0$ for some positive~$q$. The smallest~$q$ with this property is called the \textit{nilpotency index} of~$M$.

For $(T,K)$ being $(\I_g^b,\K_g^b)$ (with $g\ge3$), $(\IA_n,[\IA_n,\IA_n])$, or~$(\IO_n,[\IO_n,\IO_n])$ (with $n\ge3$), one can study the question of whether $K^{\ab}$ is a nilpotent $\Z A$-module, where $A=T/K$.
Note that nilpotency is a much more restrictive condition than being an $S_{\fX}$-torsion module. Indeed, being an $S_{\fX}$-torsion module only requires the existence of a relation of the form
\begin{equation*}
 (e_{a_1}-1)\cdots(e_{a_k}-1)M=0
\end{equation*}
with $a_i\in \fX$, while nilpotency implies that this relation holds for all~$a_i$'s, provided that $k$ is large enough. Thus, although Theorem~\ref{thm_main_gen} enables us to prove the finite generation of the group~$K^{\ab}$, it does not in any way imply the nilpotency of~$K^{\ab}$.

For $g\ge 4$, Dimca, Hain, and Papadima~\cite{DHP14} proved that the $\Q U_g$-module $H_1(\K_g,\Q)$ is nilpotent. Moreover, for $g\ge 6$, \cite[Theorem~C]{DHP14} and an explicit computation from~\cite{MSS20} yield that the nilpotency index of~$H_1(\K_g,\Q)$ is equal to~$2$. Ershov and He~\cite{ErHe18} esteblished the nilpotency of the $\Z U_g^b$-modules $(\K_g^b)^{\ab}$ for $g\ge12$ and $b\in\{0,1\}$ and the $\Z\bigl(\IA_n^{\ab}\bigr)$-modules $[\IA_n,\IA_n]^{\ab}$ for $n\ge 4$. Note, however, that the proof from~\cite{CEP22} of the fact that $\K_g^b$ is finitely generated for all $g\ge 4$ uses another method, which does not give the nilpotency of~$(\K_g^b)^{\ab}$. So the question of whether $(\K_g^b)^{\ab}$ is a nilpotent $\Z U_g^b$-module is still open for $3\le g\le 11$.

Our next result esteblishes a certain form of nilpotency property of~$\K_g^{\ab}$ for all $g\ge 3$. However, we will prove a weaker statement: nilpotency will be proved not over the ring~$\Z U_g$ itself but over its subring~$\Z V$, where $V$ is a certain finite index subgroup of~$U_g$. This result seems to be new not only for $g=3$ but also for all $g<12$, so we formulate it for an arbitrary genus. We set $$V_g=2(g-1)(g-2)U_g.$$ Then $V_g$ is a subgroup of~$U_g$ of index
$
\bigl(2(g-1)(g-2)\bigr)^{\binom{2g}{3}-2g}.
$

\begin{thm}\label{thm_nilpotent}
 Suppose that $g\ge 3$. Then $\K_g^{\ab}$ is a nilpotent $\Z V_g$-module of nilpotency index no greater than $32g^2-104g+65$.
\end{thm}

In the most important to us case $g=3$, we can replace the subgroup~$V_3=4U_3$ with a larger subgroup~$\tV_3$ of~$U_3$. To define this subgroup, consider the contraction homomorphism $C_{\mathrm{mod}\,2}\colon U_3\twoheadrightarrow H\otimes (\Z/2\Z)$ given by
$$
C_{\mathrm{mod}\,2}(x\wedge y\wedge z) = (x\cdot y)z+(y\cdot z)x+(z\cdot x)y\mod 2
$$
This homomorphism is well defined on~$U_3$. (For an arbitrary~$g$, the contraction homomorphism on~$U_g$ is well defined modulo~$g-1$, see~\cite[Lemma~7B]{Joh80b}.)
We set $\tV_3=2\ker C_{\mathrm{mod}\,2}$. Since $\rk(H)=6$ and $\rk(U_3)=14$, we have $|U_3:\ker C_{\mathrm{mod}\,2}|=2^6$, $|\ker C_{\mathrm{mod}\,2}:\tV_3|=2^{14}$, and thus $|U_3: \tV_3|=2^{20}$.

\begin{addthm}
 $\K_3^{\ab}$ is a nilpotent $\Z \tV_3$-module of nilpotency index no greater than~$41$.
\end{addthm}

\begin{remark}
 Theorem~\ref{thm_nilpotent} allows to give explicit, though enormous, bounds on the number of generators of~$\K_3^{\ab}$ and on the dimension of~$H_1(\K_3,\Q)$, see Theorem~\ref{thm_quant}.
\end{remark}

\begin{remark}
 Surprisingly, our proof of Theorem~\ref{thm_nilpotent} apparently cannot be extended to the case of the group~$\K_g^1$, see Remark~\ref{remark_not_Kg1}.
\end{remark}

The proof of Theorem~\ref{thm_nilpotent} does not rely on Theorems~\ref{thm_main_K} or~\ref{thm_main_gen} in any way. On the contrary, Theorem~\ref{thm_nilpotent} provides an alternative proof of Theorem~\ref{thm_main_K} in the case $b=0$, since it is well-known that nilpotency of a finitely generated $\Z A$-module implies its finite generation as an abelian group.

The proof of Theorem~\ref{thm_nilpotent} relies on a specific relation in~$(\K_g^b)^{\ab}$, which seems to be interesting in its own right. Suppose that $L\subset H$ is a subgroup of rank~$3$. Then we denote by~$u_L$ the image of the generator of the top exterior power~$\exter^3L$ under the map
$
\exter^3L\to U_g^b
$
induced by the inclusion $L\subset H$.
We have
$
u_L=a_1\wedge a_2\wedge a_3,
$
where $a_1,a_2,a_3$ is a basis of~$L$. The element~$u_L$ is defined up to a sign, which will usually be unimportant to us.

We denote by~$T_{\gamma}$ the Dehn twist about a simple closed curve~$\gamma$.

\begin{thm}\label{thm_explicit}
Let $\gamma$ be a  separating simple closed curve on~$\Sigma_g^b$, where $g\ge 3$ and $b\in\{0,1\}$, such that $\gamma$ is not homotopic to either a point or a boundary component, and let $H=P_1\oplus P_2$ be the corresponding decomposition in homology. Suppose that $L$ is a subgroup of~$H$ such that
\begin{itemize}
 \item $\rk L=3$,
 \item $L$ is an isotropic direct summand of~$H$,
 \item $L=(L\cap P_1)\oplus(L\cap P_2)$ and both subgroups~$L\cap P_1$ and~$L\cap P_2$ are nontrivial.
\end{itemize}
Then the following identity holds in~$(\K_g^b)^{\ab}$:
 \begin{equation}\label{eq_main_nilpot}
  (e_{u_L}+1)(e_{u_L}-1)^5\cdot [T_{\gamma}]=0.
 \end{equation}
\end{thm}

To prove Theorem~\ref{thm_explicit}, we will study functions $F\colon\Z\to(\K_g^b)^{\ab}$ of the form
$$
F(t)=\left[c^tzc^{-t}\right],
$$
where $c\in\I_g^b$ and~$z\in\K_g^b$ are certain specific fixed elements, and some more complicated similar functions of several integer variables. Then we employ simple geometric relations in the mapping class group: commutativity of twists about disjoint curves and lantern relations, to derive functional equations for such functions. Solving these equations, we arrive at relation~\eqref{eq_main_nilpot}.

\subsection{Application to the second homology}

Alongside with the abelianizations of the groups~$\K_g^b$, $[\IO_n,\IO_n]$, and $[\IA_n,\IA_n]$, one can study their higher homology groups. Very little is known about the homology groups $H_k(\K_g,\Z)$ for $g\ge3$ and~$k\ge 2$. Bestvina, Bux, and Margalit~\cite{BBM10} showed that the cohomological dimension of~$\K_g$ is equal to~$2g-3$, and then the author~\cite{Gai22} proved that the top homology group $H_{2g-3}(\K_g,\Z)$ is not finitely generated, see also the paper~\cite{Spi24} of Spiridonov, who computed explicitly an infinite rank subgroup of~$H_{2g-3}(\K_g,\Z)$ generated by abelian cycles. For intermediate dimensions $1<k<2g-3$, nothing is known about whether the groups $H_k(\K_g,\Z)$ are finitely generated. Apparently, even less is known about the homology groups $H_k([\IO_n,\IO_n],\Z)$ and $H_k([\IA_n,\IA_n],\Z)$ for $k\ge 2$.  A classical result by Krsti\'c and McCool~\cite{KrMC97} states that $\IA_3$ is not finitely presented. Further, Bestvina, Bux, and  Margalit~\cite{BBM07} proved  that the cohomological dimension of~$\IO_n$ is equal to~$2n-4$ and the top homology group $H_{2n-4}(\IO_n,\Z)$ is not finitely generated, and moreover, contains a free abelian subgroup of infinite rank. Nevertheless, the author is unaware of any results concerning the finite or infinite generation of the $k$th homology groups (with $k>1$) of the commutator subgroups of~$\IA_n$ or~$\IO_n$.

Theorem~\ref{thm_main_OA} yields the following corollary.

\begin{cor}\label{cor_IO_H2}
 The group $H_2([\IO_3,\IO_3],\Z)$ contains a free abelian subgroup of infinite rank. Moreover, the group of coinvariants $H_2([\IO_3,\IO_3],\Z)_{\IO_3}$ also contains a free abelian subgroup of infinite rank.
\end{cor}

\begin{proof}
For $n=3$, the lower row of~\eqref{eq_exact_IA} is the short exact sequence
$$
1\longrightarrow [\IO_3,\IO_3]\longrightarrow \IO_3\longrightarrow W_3\longrightarrow 1,
$$
where $W_3=(H^*\otimes\exter^2H)/\iota(H)\cong \Z^6$. Let
$$
E^2_{p,q}=H_p(W_3,H_q([\IO_3,\IO_3],\Z)) \Longrightarrow H_{p+q}(\IO_3,\Z)
$$
be the homological Lyndon--Hochschild--Serre spectral sequence corresponding to this short exact sequence. By Theorem~\ref{thm_main_OA} we have that the group $H_1([\IO_3,\IO_3],\Z)$ is finitely generated. Therefore, the group $E^2_{1,1}$, and hence the group~$E^{\infty}_{1,1}$, is finitely generated. Moreover,  $E_{2,0}^2=H_2(W_3,\Z)\cong \Z^{15}$. Consequently, the group $E^{\infty}_{2,0}$ is finitely generated. The groups~$E^{\infty}_{0,2}$, $E^{\infty}_{1,1}$, and~$E^{\infty}_{2,0}$ are associated factors for a filtration in the group~$H_2(\IO_3,\Z)$, which contains a free abelian subgroup of infinite rank by the above mentioned result from~\cite{BBM07}. Thus, the group $E^{\infty}_{0,2}$ contains a free abelian subgroup of infinite rank. The corollary follows, since  $E^{\infty}_{0,2}$ is a quotient of $E^2_{0,2}=H_2([\IO_3,\IO_3],\Z)_{\IO_3}$.
\end{proof}

It is a long-standing question whether the groups~$H_2(\I_g^b,\Z)$ for $g\ge 3$ and $b\in\{0,1\}$ are finitely generated. Recently Minahan~\cite{Min23} proved that $H_2(\I_g^b,\Q)$ is finite-dimensional for $g\ge 51$, and then Minahan and Putman~\cite{MiPu25} extended this result to all $g\ge 5$.  In genus~$3$, the question of whether the groups~$H_2(\I_3^b,\Z)$ are finitely generated is completely open, though the author~\cite{Gai21} has provided some arguments towards a possible negative answer. So we can only obtain the following conditional analogue of Corollary~\ref{cor_IO_H2}; its proof is literally the same.

\begin{cor}
 Suppose that $b\in\{0,1\}$. Then the group $H_2(\I_3^b,\Z)$ is finitely generated if and only if the group $H_2(\K_3^b,\Z)_{\I_3^b}$ is finitely generated.
\end{cor}

\subsection{Organization of the paper}

In Section~\ref{section_gen} we prove Theorem~\ref{thm_more_gen}, which provides a general sufficient condition for a finitely generated $\Z A$-module, where $A$ is a finitely generated abelian group, to be itself finitely generated as an abelian group. Section~\ref{section_Johnson} contains necessary preliminary information on the Johnson homomorphisms, both for mapping class groups and for automorphisms of free groups. In Section~\ref{section_SDP} we esteblish the subgroup displacement property for the groups under consideration, that is, for the actions of~$\Sp_{2g}(\Z)$ on~$U_g^1$ and~$U_g$ and for the actions of~$\SL_n(\Z)$ on~$\IA^{\ab}_n$ and~$\IO^{\ab}_n$.

Now, to prove Theorems~\ref{thm_main_K} and~\ref{thm_main_OA} we suffice to show that $(\K_g^b)^{\ab}$ for $g\ge3$ and $[\IO_n,\IO_n]^{\ab}$ for $n\ge3$ are $S_{\fX}$-torsion modules for appropriate subsets~$\fX$. This will be done in Sections~\ref{section_tors_prelim}--\ref{section_proof_IO}. Namely, in Section~\ref{section_tors_prelim} we prove several general lemmas allowing us to show that some elements of modules are $S_{\fX}$-torsion. In Section~\ref{section_proof_K} we esteblish the property of being $S_{\fX}$-torsion for the module $(\K_g^b)^{\ab}$ and then prove Theorem~\ref{thm_main_K}. Similarly, in Section~\ref{section_proof_IO} we esteblish the property of being $S_{\fX}$-torsion for the modules $[\IO_n,\IO_n]^{\ab}$ and $[\IA_n,\IA_n]^{\ab}$ and then prove Theorem~\ref{thm_main_OA}.

In Sections~\ref{section_equations}--\ref{section_proof_explicit} we prove  Theorem~\ref{thm_explicit}. Section~\ref{section_equations} is auxiliary. In it we consider functions $f\colon A\to B$, where $A$ is a free abelian group of finite rank and $B$ is an arbitrary abelian group, define an important for us class of quasipolynomial functions $\Z\to B$, and then study and  partially solve four specific functional equations. In Section~\ref{section_config} we construct a special configuration of simple closed curves on~$\Sigma_g^b$ that will play a key role in our proof of Theorem~\ref{thm_explicit}. In Section~\ref{section_sF} we introduce certain special elements $s_i(t,\bm,\bn)\in\K_g^b$, where $i\in\{1,2,3\}$, $t\in \Z$, $\bm,\bn\in\Z^3$; then we define functions $F_i\colon\Z^7\to(\K_g^b)^{\ab}$ by $F_i(t,\bm,\bn)=[s_i(t,\bm,\bn)]$ and use geometric relations in the mapping class group to derive functional equations on~$F_i$. Further, in Section~\ref{section_quasi} we study these functional equations using the results from Section~\ref{section_equations} and esteblish the quasipolynomial property for~$F_i$. Finally, in Section~\ref{section_proof_explicit}, we prove Theorem~\ref{thm_explicit}.

In Sections~\ref{section_subgroups} and~\ref{section_proof_nilpotent} we deduce Theorem~\ref{thm_nilpotent} from Theorem~\ref{thm_explicit} and also obtain explicit, albeit enormous, bounds on the number of generators of~$\K_3^{\ab}$ and on the dimension of~$H_1(\K_3,\Q)$. Section~\ref{section_open} contains some remarks and open questions.

\smallskip

The author is grateful to Ilya Bogdanov, Alexander Efimov, and Anton Fonarev for fruitful discussions.

\section{Proof of general sufficient conditon for finite generation}\label{section_gen}

In this section we prove Theorem~\ref{thm_more_gen}.
Obiviously, we suffice to prove this theorem in the case of a $\Z A$-module~$M$ generated by a single element~$x$, i.\,e., to prove the following proposition.

\begin{propos}
 Let $A$ be a finitely generated abelian group and $M$ a $\Z A$-module generated by a single element~$x$. Assume that the following condition holds:
 \begin{itemize}
  \item[$(*)$] For each subgroup $B\subset A$ with $|A: B|=\infty$, there exists a positive integer~$k$ and elements $a_1,\ldots,a_k\in A\setminus B$ such that
  \begin{equation*}
   (e_{a_1} -1)\cdots (e_{a_k} -1)x=0.
  \end{equation*}
 \end{itemize}
 Then $M$ is a finitely generated abelian group.
\end{propos}

\begin{proof}
 Let us prove the statement of the proposition by the induction on $r=\rk (A)$. (Under the \textit{rank} of a finitely generated abelian group~$A$ we always mean the rank of its torsion-free component, i.\,e., the dimension of the vector space~$A\otimes\Q$.) If $r=0$, then the group~$A$ is finite, so all finitely generated $\Z A$-modules are finitely generated abelian groups, hence the base case of induction.

 Suppose that $r>0$. Let $A_{\mathrm{tors}}$ be the torsion subgroup of~$A$. Property~$(*)$ for $B=A_{\mathrm{tors}}$ says that there exists a~$k$ and elements $a_1,\ldots,a_k\in A$ of infinite order such that
 \begin{equation}\label{eq_rel_dtl}
   (e_{a_1} -1)\cdots (e_{a_k} -1)x=0.
 \end{equation}
 The smallest length~$k$ of such a relation will be referred to as the \textit{difference torsion length} of~$x$. For each fixed~$r$, we will prove the statement of the proposition by the induction on the difference torsion index~$k$ of~$x$. The base of induction for $k=0$ is trivial.

 Suppose that $k>0$. Consider the relation~\eqref{eq_rel_dtl}. Let $N$ be the $\Z A$-submodule of~$M$ generated by the element $y=(e_{a_1}-1)x$. The difference torsion length of~$y$ is equal to~$k-1$. So by the induction hypothesis we have that $N$ is a finitely generated abelian group.

 Now, consider the quotient module $M/N$ and the quotient group $A_1=A/\langle a_1\rangle$, where $\langle a_1\rangle$ denotes the subgroup of~$A$ generated by~$a_1$. Since $e_{a_1}$ acts trivially on~$M/N$, we see that $M/N$ is a $\Z A_1$-module. Denote by~$x_1$ the image of~$x$ in~$M/N$; then $x_1$ generates the $\Z A_1$-module~$M/N$. Property~$(*)$ for the pair~$(A_1,x_1)$ follows immediately from property~$(*)$ for the pair~$(A,x)$. Moreover, we have $\rk (A_1)=r-1$, since $a_1$ is an element of infinite order of~$A$. So by the induction hypothesis we have that $M/N$ is a finitely generated abelian group.

Thus, $M$ is also a finitely generated abelian group.
\end{proof}

\section{Preliminaries on Johnson homomorphisms}\label{section_Johnson}

In this section we recall several standard facts about Johnson homomorphisms both for the Torelli subgroups~$\I_g^b$ of the mapping class groups~$\Mod(\Sigma_g^b)$ and the Torelli subgroups~$\IA_n$ and~$\IO_n$ of the automorphism groups~$\Aut(F_n)$ and~$\Out(F_n)$, respectively.

\subsection{Mapping class groups}\label{subs_MCG}
Throughout the paper, we denote by~$T_{\gamma}$ the left Dehn twist about a simple closed curve~$\gamma$. For an oriented simple closed curve~$\gamma$, we denote by~$[\gamma]$ the integral homology class of~$\gamma$. All necessary basic facts about curves on surfaces and mapping class groups can be found in the book~\cite{FaMa12}.

We always assume that $g\ge 3$ and $b\in\{0,1\}$.
As in the Introduction, let $H=H_1\bigl(\Sigma_g^b,\Z\bigr)$. Then $H\cong\Z^{2g}$ is a tautological $\Sp_{2g}(\Z)$-module. Let $H_{\bC}=H\otimes\bC$. We have the following commutative diagram (see~\cite{Joh80b}) in which the horizontal arrows~$\tau$ are $\Mod(\Sigma_g^b)$-equivariant and are called \textit{Johnson homomorphisms}:
\begin{equation*}
 \begin{tikzcd}
  \I_g^1 \arrow[r, two heads, "\tau"] \arrow[d, two heads] &
  \exter^3H \arrow[r, equal] \arrow [d, two heads, "\pi"] & \!\!\!\colon U_g^1\\
  \I_g \arrow[r, two heads, "\tau"] &
  \exter^3H/(\Omega\wedge H) \arrow[r, equal] & \!\!\!\colon U_g
 \end{tikzcd}
\end{equation*}
The complexification $U_g^1\otimes\bC=\exter^3H_{\bC}$ is a rational representation of~$\Sp_{2g}(\bC)$. However, this representation is not irreducible. Namely, we have that $U_g^1\otimes\bC\cong[1]\oplus[1^3]$, see~\cite[Chapter~17]{FuHa91}. Hereafter, $[\lambda_1\ldots\lambda_m]$ denotes the representation corresponding to the Young diagram $[\lambda_1\ldots\lambda_m]$. In particular, $[1^3]$ corresponds to the Young diagram consisting of $3$ boxes in a single column. Note that the summand~$[1]$ of~$U_g^1\otimes\bC$ is exactly the kernel of the projection~$\pi_{\bC}=\pi\otimes\bC$. The other summand~$[1^3]$ is the kernel of the complexification of the contraction homomorphism $C\colon\exter^3H\to H$ given by
\begin{equation*}
C(x\wedge y\wedge z) = (x\cdot y)z+(y\cdot z)x+(z\cdot x)y,
\end{equation*}
see~\cite{Joh80b}. Thus, we arrive at the following proposition.

\begin{propos}\label{propos_U_decompose}
We have
\begin{equation*}
 U_g^1\otimes\bC=\ker\pi_{\bC}\oplus\ker C_{\bC}.
\end{equation*}
Both summands are irreducible rational $\Sp_{2g}(\bC)$-representations.
\end{propos}

\begin{cor}\label{cor_U_irr}
 The $\Sp_{2g}(\bC)$-representation $U_g\otimes\bC$ is irreducible. Namely, $U_g\otimes\bC\cong[1^3]$.
\end{cor}

Next, we need two formulas for the values of~$\tau$ on certain specific elements of~$\I_g^b$. First, recall that a \textit{bounding pair} is a pair~$\{\gamma_1,\gamma_2\}$ of simple closed curves on~$\Sigma_g^b$ such that
\begin{itemize}
 \item $\gamma_1$ and $\gamma_2$ are disjoint and not homotopic to each other,
 \item each of the curves~$\gamma_i$ does not separate~$\Sigma_g^b$,
 \item the union~$\gamma_1\cup\gamma_2$ divides~$\Sigma_g^b$ into two parts.
\end{itemize}
The mapping class~$T_{\gamma_1}T_{\gamma_2}^{-1}$ is called a \textit{bounding pair map} (or simply a \textit{BP map}); it belongs to the Torelli group~$\I_g^b$. The following result is due to Johnson~\cite[Lemma~4B]{Joh80b}.

\begin{propos}
 Suppose that $\{\gamma_1,\gamma_2\}$ is a bounding pair in~$\Sigma_g^1$. Let $S$ be the connected component of\/~$\Sigma_g^1\setminus(\gamma_1\cup\gamma_2)$ that does not contain the boundary~$\partial \Sigma_g^1$ and let $g'$ be the genus of~$S$. Orient the curve~$\gamma_1$ so that $S$ lies to its right, and denote  by~$c$ the homology class of~$\gamma_1$, see Fig.~\ref{fig_BP}. Then
 \begin{gather}\label{eq_Joh_BP}
  \tau\left(T_{\gamma_1}T_{\gamma_2}^{-1}\right)=\left(\sum_{i=1}^{g'}a_i\wedge b_i\right)\wedge c,\\
  C\left(\tau\left(T_{\gamma_1}T_{\gamma_2}^{-1}\right)\right)=g'c.\nonumber
 \end{gather}
 where $a_1,b_1,\ldots,a_{g'},b_{g'}$ is a symplectic basis of an arbitrary nondegenerate symplectic submodule $P\subset H_1(S,\Z)$ of rank~$2g'$.
\end{propos}

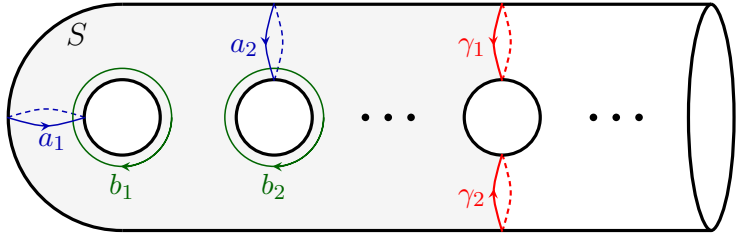
\begin{figure}
\begin{center}
\begin{tikzpicture}[scale=.5,
        mid arrow/.style={
            postaction={
                decorate,
                decoration={
                    markings,
                    mark=at position 0.55 with {
                        \arrow[#1]{stealth}
                    }
                }
            }
        },
        mid arrow/.default=black,
    ]

\small
\definecolor{myblue}{rgb}{0, 0, 0.7}
\definecolor{mygreen}{rgb}{0, 0.4, 0}
\tikzset{every path/.append style={line width=.2mm}}
\tikzset{my dash/.style={dash pattern=on 2pt off 1.5pt}}

\fill [color=black!4] (6,-3.03) .. controls (5.7,-2) .. (6,-.97) arc (270:90:.97) .. controls (5.7,2) .. (6,3.03) -- (-4,3.03) arc(90:270:3.03) -- (6,-3.03);

\draw[color=myblue,my dash] (-7.03,0) .. controls (-6,0.3) .. (-4.97,0); 
\draw[color=myblue,my dash] (0,.97) .. controls (0.3,2) .. (0,3.03); 
\draw[color=red, thick, my dash] (6,3.03) .. controls (6.3,2) .. (6,.97); 

\draw[color=red, thick, my dash] (6,-3.03) .. controls (6.3,-2) .. (6,-.97); 


\draw [very thick,fill=white] (0,0) circle (1);
\draw [very thick,fill=white] (-4,0) circle (1);
\fill (2.4,0) circle (.1);
\fill (3,0) circle (.1);
\fill (3.6,0) circle (.1);
\draw [very thick] (6,0) circle (1);
\fill (8.4,0) circle (.1);
\fill (9,0) circle (.1);
\fill (9.6,0) circle (.1);
\draw [very thick]  (-4,3)--(11.5,3);
\draw [very thick]  (-4,-3)--(11.5,-3);
\draw [very thick] (-4,3) arc (90:270:3); 
\draw [very thick] (11.5,3) .. controls (11.7,3) and (12.1,2) .. (12.1,0)
.. controls (12.1,-2) and (11.7,-3) .. (11.5,-3)
.. controls (11.3,-3) and (10.9,-2) .. (10.9,0)
.. controls (10.9,2) and (11.3,3) .. (11.5,3);


\draw [color=mygreen,-stealth] (-2.7,0) arc (0:-452:1.3)node [pos=.2, below=-1pt] {$b_1$} ;

\draw [color=mygreen,-stealth] (1.3,0) arc (0:-452:1.3) node [pos=.2, below=-1pt] {$b_2$} ;

\draw[color=myblue,mid arrow=myblue] (-7.03,0) .. controls (-6,-0.3) .. (-4.97,0)  node [pos=0.6,below=-1pt] {$a_1$};
\draw[color=myblue,mid arrow=myblue] (0,3.03) .. controls (-0.3,2) .. (0,.97) node [pos=0.6,left=-1pt] {$a_2$}; 

\draw[color=red, thick, mid arrow=red] (6,3.03) .. controls (5.7,2) .. (6,.97) node [pos=0.6,left=-1pt] {$\gamma_1$}; 

\draw[color=red, thick, mid arrow=red] (6,-3.03) .. controls (5.7,-2) .. (6,-.97) node [pos=0.4,left=-1pt] {$\gamma_2$}; 




\draw(-5.2,2.2) node {\normalsize$S$};

\end{tikzpicture}
\end{center}
 \caption{Bounding pair map}\label{fig_BP}
\end{figure}

\begin{cor}\label{cor_I1_tau}
If $g'<g$, then $\tau\left(T_{\gamma_1}T_{\gamma_2}^{-1}\right)$ belongs neither to~$\ker\pi$ nor to~$\ker C$.
\end{cor}

\begin{remark}
 If $g'=g$, then $\tau\left(T_{\gamma_1}T_{\gamma_2}^{-1}\right)$ belongs to~$\ker\pi$.
\end{remark}

\begin{cor}\label{cor_I0_tau}
For a closed surface~$\Sigma_g$, the same formula~\eqref{eq_Joh_BP} holds for either connected component~$S$ of\/~$\Sigma_g\setminus(\gamma_1\cup\gamma_2)$, where the expression on the right-hand side is considered as an element of~$U_g$. In this case, we always have $g'<g$, so the value $\tau\left(T_{\gamma_1}T_{\gamma_2}^{-1}\right)$ is nonzero in~$U_g$.
\end{cor}

Second, recall that a \textit{simply intersecting pair} is a pair~$\{\alpha_1,\alpha_2\}$ of simple closed curves on~$\Sigma_g^b$ whose geometric intersection number is~$2$ and whose algebraic intersection number is~$0$. Then $(T_{\alpha_1},T_{\alpha_2})\in\I_g^b$. These elements of Torelli groups are called \textit{commutators of simply intersecting pairs}, see~\cite{Put07}. Choose orientations for the curves~$\alpha_1$ and~$\alpha_2$ arbitrarily, and let $a_1$ and~$a_2$ be the homology classes of the obtained oriented curves, respectively. Denote the intersection points of~$\alpha_1$ and~$\alpha_2$ by~$p_+$ and~$p_-$ so that the intersection number of~$\alpha_1$ and~$\alpha_2$ at~$p_+$ is~$+1$ and the intersection number of~$\alpha_1$ and~$\alpha_2$ at~$p_-$ is~$-1$. Let $\alpha_3$ be the oriented simple closed curve that passes from~$p_-$ to~$p_+$ along the positively oriented arc of~$\alpha_1$ and then returns to~$p_-$ along the positively oriented arc of~$\alpha_2$, see Fig.~\ref{fig_SIP}, and let $a_3$ be the homology class of~$\alpha_3$. The following formula for the value of~$\tau$ on the commutator of a simply intersecting pair was obtained independently by Childers~\cite[Proposition~3.4]{Chi12}, Church~\cite[Proposition~6.4]{Chu14}, and Putman~\cite{Put18}\footnote{This result is contained in the preliminary preprint version of~\cite{Put18}. Nevertheless, it did not enter the final journal version of that paper.}.

\begin{figure}
\begin{center}
\begin{tikzpicture}[scale=.5]
\small
\definecolor{myblue}{rgb}{0, 0, 0.7}
\definecolor{mygreen}{rgb}{0, 0.4, 0}
\tikzset{every path/.append style={line width=.2mm}}
\tikzset{my dash/.style={dash pattern=on 2pt off 1.5pt}}



\draw [very thick] (0.3,0) circle (.9);
\draw [very thick] (4.6,0) circle (.9);
\draw [very thick] (8.9,0) circle (.9);
\fill (12,0) circle (.1);
\fill (12.6,0) circle (.1);
\fill (13.2,0) circle (.1);
\draw [very thick] (-4,0) circle (.9);
\draw [very thick]  (-4,3)--(13.5,3);
\draw [very thick]  (-4,-3)--(13.5,-3);
\draw [very thick] (-4,3) arc (90:270:3); 

\draw [mygreen] (-.1,0) circle (1.5);
\draw [mygreen,-stealth] (-.1,1.5) arc (90:135:1.5);



\draw [color=red] (0,1.8) arc (90:270:1.8) -- (4.6,-1.8) node [pos=.75,below=-1pt] {$\alpha_2$} arc (-90:90:1.8) -- (0,1.8) ;
\draw [color=myblue] (-4,1.6) arc (90:270:1.6) -- (0,-1.6) node [pos=.15,below=-1pt] {$\alpha_1$} arc (-90:90:1.6) -- (-4,1.6) ;
\draw [color=red,-stealth] (3,1.8)--(2,1.8);
\draw [color=myblue,-stealth] (-1,1.6)--(-2.7,1.6);

\fill (-.81,1.6) circle (.08) node [above=-1pt] {$p_+$};
\fill (-.81,-1.6) circle (.08) node [below=-1pt] {$p_-$};

\draw(-1.6,0) node [right=-3pt, mygreen] {$\alpha_3$};




\end{tikzpicture}
\end{center}
 \caption{Commutator of a simply intersecting pair}\label{fig_SIP}
\end{figure}
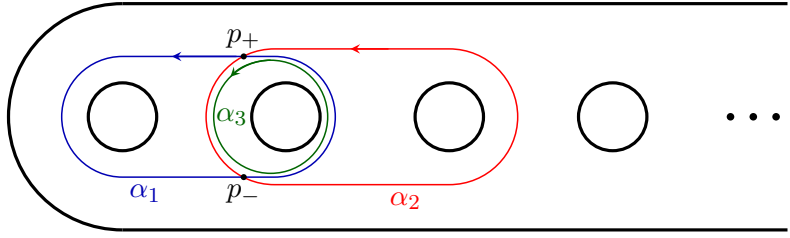

\begin{propos}\label{propos_csip}
We have $\tau\bigl((T_{\alpha_1},T_{\alpha_2})\bigr)=a_1\wedge a_2\wedge a_3$.
\end{propos}

\begin{remark}
 The sign in this formula is opposite to the sign in~\cite{Chu14}, since we are using left Dehn twists rather than right ones.
\end{remark}

\subsection{Automorphisms of free groups}\label{subs_Aut}
We always assume that $n\ge 3$.
Now, let $$H=H_1(F_n,\Z)=F_n^{\ab}.$$ Then $H\cong\Z^{n}$ is a tautological $\SL_{n}(\Z)$-module. Let $H_{\bC}=H\otimes\bC$. We have the following commutative diagram in which the horizontal arrows are called \textit{Johnson homomorphisms}, see~\cite{For90}:
\begin{equation*}
 \begin{tikzcd}
  \IA_n \arrow[r, two heads, "\tau"] \arrow[d, two heads] &
  H^*\otimes\exter^2H  \arrow [d, two heads, "\pi"] \arrow[r, equal] & \!\!\!\colon \widetilde{W}_n
  \\
  \IO_n \arrow[r, two heads, "\tau"] &
  (H^*\otimes\exter^2H)/\iota(H) \arrow[r, equal] & \!\!\!\colon W_n
 \end{tikzcd}
\end{equation*}
Recall that the embedding $\iota\colon H\hookrightarrow H^*\otimes\exter^2H$ is induced by the canonical embedding $\Z\hookrightarrow H^*\otimes H=\Hom(H,H)$ that takes~$1$ to the identity operator.  In other words, if $e_1,\ldots,e_n$ is a basis of~$H$ and $e_1^*,\ldots,e_n^*$ is the dual basis of~$H^*$, then $\iota(v)=\sum_{i=1}^ne_i^*\otimes(e_i\wedge v)$. The complexification $\widetilde{W}_n\otimes\bC=H^*_{\bC}\otimes\exter^2H_{\bC}$ is a rational representation of~$\SL_{n}(\bC)$. However, this representation is not irreducible. Namely, we have the following analogue of Proposition~\ref{propos_U_decompose}, see~\cite{For90}.

\begin{propos}\label{propos_W_decompose}
We have
\begin{equation}\label{eq_tW_decompose}
 \widetilde{W}_n\otimes\bC=\ker\pi_{\bC}\oplus\ker C_{\bC},
\end{equation}
where $C\colon H^*\otimes\exter^2H\to H$ is the contraction homomorphism given by
$$
C\bigl(\xi\otimes(x\wedge y)\bigr)=\xi(x)y-\xi(y)x.
$$
Both summands in~\eqref{eq_tW_decompose} are irreducible rational $\SL_n(\bC)$-representations.
\end{propos}

\begin{cor}\label{cor_W_irr}
 The $\SL_n(\bC)$-representation $W_n\otimes\bC$ is irreducible and is isomorphic to~$\ker C_{\bC}$.
\end{cor}

A classical theorem due to Nielsen~\cite{Nie24} for $n\le 3$ and to Magnus~\cite{Mag35} for all~$n$ (cf.~\cite{MKS}) says that the group~$\IA_n$ is generated by the following automorphisms~$K_{ij}$ and~$K_{ijk}$:
\begin{align*}
 K_{ij}&\colon
 \left\{
 \begin{aligned}
  a_i&\mapsto a_ja_ia_j^{-1},\\
  a_r&\mapsto a_r,\quad r\ne i,
 \end{aligned}
 \right.&
 K_{ijk}&\colon
 \left\{
 \begin{aligned}
  a_i&\mapsto a_i(a_j,a_k),\\
  a_r&\mapsto a_r,\quad r\ne i,
 \end{aligned}
 \right.
\end{align*}
where the subsripts of each homomorphism are distinct numbers from~$\{1,\ldots,n\}$. Our convention is that automorphisms of~$F_n$ are applied on the right so that, for instance, $K_{12}K_{13}$ takes~$a_1$ to~$a_2a_3a_1a_3^{-1}a_2^{-1}$. The following computation is also due to Formanek~\cite{For90}.

\begin{propos}
 We have $\tau(K_{ij})=e_i^*\otimes (e_i\wedge e_j)$ and $\tau(K_{ijk})=e_i^*\otimes (e_j\wedge e_k)$.
\end{propos}

The reason our approach fails for the group~$[\IA_3,\IA_3]^{\ab}$ is that in this case we lack an analogue of Corollary~\ref{cor_I1_tau}. Namely, we have the following statement.
\begin{cor}\label{cor_IA_tau}
 The elements $\tau(K_{ij})$ belong neither to~$\ker\pi$ nor to~$\ker C$. However, the elements~$\tau(K_{ijk})$ belong to the summand~$\ker C$.
\end{cor}

Nevertheless, we still have the following analogue of Corollary~\ref{cor_I0_tau}, which will allow us to handle the case of the group~$[\IO_3,\IO_3]^{\ab}$.

\begin{cor}\label{cor_IO_tau}
 The elements $\pi(\tau(K_{ij}))$ and~$\pi(\tau(K_{ijk}))$ are all nonzero in~$W_n$.
\end{cor}

\section{Subgroup displacement property}\label{section_SDP}

\begin{propos}
\label{propos_SDP}
Suppose that $G$ is an irreducible complex algebraic group, $D$ is a Zariski dense subgroup of~$G$, and $A$ is a finitely generated free abelian group. Assume that $D$ acts on~$A$ via a homomorphism $D\to \GL(A)$ so that this action extends to a rational $G$-representation in~$A\otimes\bC$. Let $\fX\subset A\setminus\{0\}$ be the subset consisting of all elements  that are not contained in any proper $G$-invariant subspace of~$A\otimes\bC$. Then the action of~$D$ on~$A$ possesses the subgroup displacement property with respect to~$\fX$.
\end{propos}

\begin{proof}
 Let $B$ be a subgroup of~$A$ such that $\rk(B)<\rk(A)$. Consider an element $a\in \fX$. Since $a$ is not contained in any proper $G$-invariant subspace of~$A\otimes\bC$, we obtain that the $G$-orbit of~$a$ is not contained in~$B\otimes\bC$. Let $O_a\subset G$ be the subset consisting of all elements $g$ such that $g(a)\notin B\otimes\bC$. Then $O_a$ is a nonempty Zariski open subset of~$G$. Therefore, for a finite subset $X\subset \fX$, the intersection $O_X=\bigcap_{a\in X}O_a$ is also a nonempty Zariski open subset of~$G$. Since $D$ is Zariski dense in~$G$, the subset~$O_X$ contains an element of~$D$, which completes the proof of the proposition.
\end{proof}

Combining this proposition with Proposition~\ref{propos_U_decompose}, Corollary~\ref{cor_U_irr}, Proposition~\ref{propos_W_decompose}, and Corollary~\ref{cor_W_irr}, we immediately obtain the following two corollaries.

\begin{cor}\label{cor_SDP_I}
 Suppose that $g\ge 3$. Let $U_g^1$, $U_g$, $\pi$, and~$C$ be as in Subsection~\ref{subs_MCG}. Then:
 \begin{enumerate}
  \item The action of~$\Sp_{2g}(\Z)$ on~$U_g$ possesses the subgroup displacement property with respect to~$U_g\setminus\{0\}$.
  \item The action of~$\Sp_{2g}(\Z)$ on~$U_g^1$ possesses the subgroup displacement property with respect to the subset $\fX\subset U_g^1\setminus\{0\}$ consisting of all elements that lie neither in~$\ker\pi$ nor in~$\ker C$.
 \end{enumerate}
\end{cor}

\begin{cor}\label{cor_SDP_IA}
 Suppose that $n\ge 3$. Let $W_n$, $\widetilde{W}_n$, $\pi$, and~$C$ be as in Subsection~\ref{subs_Aut}. Then:
 \begin{enumerate}
  \item The action of~$\SL_n(\Z)$ on~$W_n$ possesses the subgroup displacement property with respect to~$W_n\setminus\{0\}$.
  \item The action of~$\SL_n(\Z)$ on~$\widetilde{W}_n$ possesses the subgroup displacement property with respect to the subset $\fX\subset \widetilde{W}_n\setminus\{0\}$ consisting of all elements that lie neither in~$\ker\pi$ nor in~$\ker C$.
 \end{enumerate}
\end{cor}

\section{$S_{\fX}$-torsion property for the abelianization of a subgroup}\label{section_tors_prelim}

Let $K$ be a normal subgroup of a finitely generated group~$T$ such that $A=T/K$ is free abelian. We denote by~$\tau$ the  quotient homomorphism
$T\to A$. Recall that, for a subset~$\fX\subset A\setminus\{0\}$, $S_{\fX}$ is the multiplicatively closed subset of~$\Z A$ generated by all elements $e_a-1$ such that $a\in \fX$. In this section we prove several general lemmas which will allow us to esteblish the $S_{\fX}$-torsion property for elements of~$K^{\ab}$.

\begin{lem}\label{lem_tors_comm1}
 Let $h$ be an element of~$K$ and~$y$ an element of~$T$. Suppose that $hy=yh$ and $\tau(y)\in\fX$. Then $[h]$ is an $S_{\fX}$-torsion element of~$K^{\ab}$.
\end{lem}

\begin{proof}
 We have,
 $$
 \bigl(e_{\tau(y)}-1\bigr)\cdot [h]=(y-1)\cdot [h]=\bigl[yhy^{-1}\bigr]-[h]=0.
 $$
 The lemma follows, since $e_{\tau(y)}-1\in S_{\fX}$.
\end{proof}

\begin{remark}
 This lemma is already sufficient to prove Theorem~\ref{thm_main_K}, see Section~\ref{section_proof_K}. The following lemmas will be used in the  proof of Theorem~\ref{thm_main_OA}.
\end{remark}

\begin{lem}\label{lem_tors_identity}
 Suppose that $x,y_1,\ldots,y_m$ are elements of~$T$.
Then the following identity holds in~$K^{\ab}$:
\begin{equation*}
 \bigl[(x,y_1\cdots y_m)\bigr]=\sum_{i=1}^m e_{\tau(y_1)+\cdots+\tau(y_{i-1})}\cdot\bigl[(x,y_i)\bigr].
\end{equation*}
\end{lem}

\begin{proof}
A standard commutator identity yields that
\begin{equation*}
 (x,y_1\cdots y_m)=(x,y_1)\,\ {}^{y_1}(x,y_2)\,\ {}^{y_1y_2}(x,y_3)\cdots\,{}^{y_1\cdots y_{m-1}}(x,y_m),
\end{equation*}
where ${}^y\!z=yzy^{-1}$.
The lemma follows.
\end{proof}

\begin{cor}\label{cor_tors_gen}
Suppose that $K=[T,T]$ and elements $y_1,y_2,\ldots$ generate~$T$. Assume that $\bigl[(y_i,y_j)\bigr]$ are $S_{\fX}$-torsion elements of~$K^{\ab}$ for all $i$ and~$j$. Then $K^{\ab}$  is an $S_{\fX}$-torsion module.
\end{cor}

\begin{cor}\label{cor_tors_identity}
 Suppose that $x,y_1,\ldots,y_m,z_1,\ldots,z_n$ are elements of\/~$T$ such that
 \begin{equation*}
  y_1\cdots y_m=z_1\cdots z_n.
 \end{equation*}
Then the following identity holds in~$K^{\ab}$:
\begin{equation}\label{eq_tors_identity}
 \sum_{i=1}^m e_{\tau(y_1)+\cdots+\tau(y_{i-1})}\cdot\bigl[(x,y_i)\bigr]=
 \sum_{j=1}^n e_{\tau(z_1)+\cdots+\tau(z_{j-1})}\cdot\bigl[(x,z_j)\bigr].
\end{equation}
\end{cor}

\begin{lem}\label{lem_tors_comm2}
Let $x$, $y$, and~$z$ be elements of~$T$ such that $y$ and~$z$ commute and $\tau(y)\in\fX$. Assume that $\bigl[(x,y)\bigr]$ is an $S_{\fX}$-torsion element of~$K^{\ab}$. Then $\bigl[(x,z)\bigr]$ is also an $S_{\fX}$-torsion element of~$K^{\ab}$.
\end{lem}

\begin{proof}
 Equation~\eqref{eq_tors_identity} corresponding to the relation $yz=zy$ reads as
 $$
 \bigl(e_{\tau(y)}-1\bigr)\cdot \bigl[(x,z)\bigr]=
 \bigl(e_{\tau(z)}-1\bigr)\cdot \bigl[(x,y)\bigr].
 $$
 The lemma follows, since $e_{\tau(y)}-1\in S_{\fX}$.
\end{proof}

\section{Proof of Theorem~\ref{thm_main_K}}\label{section_proof_K}

The \textit{genus} of a bounding pair~$\{\gamma_1,\gamma_2\}$ on~$\Sigma_g^1$ is, by definition, the genus of the subsurface bounded by~$\gamma_1\cup\gamma_2$, that is, the genus of those connected component of~$\Sigma_g^1\setminus(\gamma_1\cup\gamma_2)$ that does not contain the boundary~$\partial\Sigma_g^1$. For the closed surface~$\Sigma_g$, the \textit{genus} of a bounding pair~$\{\gamma_1,\gamma_2\}$ is defined to be the smallest of the genera of the two connected components of~$\Sigma_g\setminus(\gamma_1\cup\gamma_2)$.

\begin{propos}\label{propos_K_torsion}
 Suppose that $g\ge 3$, $b\in\{0,1\}$, and $0<g'<g$. Let $\fX$ be any subset of~$U_g^b\setminus\{0\}$ that contains the elements $\tau\left(T_{\gamma_1}T_{\gamma_2}^{-1}\right)$ for all genus~$g'$ bounding pairs~$\{\gamma_1,\gamma_2\}$ on~$\Sigma_g^b$. Then $(\K_g^b)^{\ab}$ is an $S_{\fX}$-torsion $\Z U_g^b$-module.
\end{propos}

\begin{proof}
 By a classical result of Johnson~\cite{Joh85a} the group $\K_g^b$ is generated by twists~$T_{\gamma}$ about separating simple closed curves~$\gamma$ on~$\Sigma_g^b$. So we suffice  to prove that the class~$[T_{\gamma}]$ of any such $T_{\gamma}$ is an $S_{\fX}$-torsion element of~$(\K_g^b)^{\ab}$. However, for each separating simple closed curve~$\gamma$, there exists a genus~$g'$ bounding pair~$\{\gamma_1,\gamma_2\}$ that is disjoint from~$\gamma$. Then the corresponding BP map $T_{\gamma_1}T_{\gamma_2}^{-1}$ commutes with~$T_{\gamma}$. The proposition now follows from Lemma~\ref{lem_tors_comm1}.
\end{proof}

Combining Proposition~\ref{propos_K_torsion} with Corollaries~\ref{cor_I1_tau} and~\ref{cor_I0_tau}, we obtain the following statement.

\begin{cor}\label{cor_tors_I}
 Suppose that $g\ge 3$. Then the group $\K_g^{\ab}$ is an $S_{U_g\setminus\{0\}}$-torsion $\Z U_g$-module.  The group $(\K_g^1)^{\ab}$ is an $S_{\fX}$-torsion $\Z U_g^1$-module for the subset $\fX\subset U_g^1\setminus\{0\}$ consisting of all elements that lie neither in~$\ker\pi$ nor in~$\ker C$.
\end{cor}

Theorem~\ref{thm_main_K} follows from Corollaries~\ref{cor_main_gen}, \ref{cor_SDP_I}, and~\ref{cor_tors_I}.

\section{Proof of Theorem~\ref{thm_main_OA}}\label{section_proof_IO}

The aim of this section is to prove the following proposition and then deduce Theorem~\ref{thm_main_OA} from it. We use notation from Subsection~\ref{subs_Aut}.

\begin{propos}\label{propos_IA_torsion}
Suppose  that $n\ge 3$. If $\fX$ is an arbitrary subset of~$\widetilde{W}_n\setminus\{0\}$ that contains all elements~$\tau(K_{ij})$ and all elements~$\tau(K_{ijk})$, then $[\IA_n,\IA_n]^{\ab}$ is an $S_{\fX}$-torsion $\Z \widetilde{W}_n$-module. Similarly, if~$\fX$ is an arbitrary subset of~$W_n\setminus\{0\}$ that contains all elements~$\pi(\tau(K_{ij}))$ and all elements~$\pi(\tau(K_{ijk}))$, then $[\IO_n,\IO_n]^{\ab}$ is an $S_{\fX}$-torsion $\Z W_n$-module.
\end{propos}

The proof of this proposition is literally identical in the cases of~$[\IA_n,\IA_n]^{\ab}$ and $[\IO_n,\IO_n]^{\ab}$, so we treat both cases simultaneously. To do this, we conveniently denote the images of the  generators~$K_{ij}$ and~$K_{ijk}$ in~$\IO_n$ again by~$K_{ij}$ and~$K_{ijk}$, respectively. Further we always assume that $\fX$ is a subset of~$\widetilde{W}_n\setminus\{0\}$ (or~$W_n\setminus\{0\}$) that satisfy the conditions from Proposition~\ref{propos_IA_torsion}.

By Corollary~\ref{cor_tors_gen} to prove Proposition~\ref{propos_IA_torsion} we only need to show that the classes in~$[\IA_n,\IA_n]^{\ab}$ (or~~$[\IO_n,\IO_n]^{\ab}$) of the commutators of all pairs of the Magnus generators~$K_{ij}$ and~$K_{ijk}$ are $S_{\fX}$-torsion elements. Also we have $K_{ikj}=K_{ijk}^{-1}$, so we can take only one of each pair of generators~$K_{ijk}$ and~$K_{ikj}$.

If the sets of subscripts of two Magnus generators are disjoint, then these generators commute. Moreover, $K_{ik}$ commutes with $K_{jk}$, $K_{ikl}$ commutes with $K_{jk}$, $K_{ikl}$ commutes with $K_{jkl}$, and $K_{ikl}$ commutes with $K_{jkm}$, where in all cases the subscripts denoted by different letters are assumed to be different. Let us consider in turn the remaining commutators of pairs of generators.

Since $\tau(K_{ik})\in\fX$, the following proposition is a direct consequence of Lemma~\ref{lem_tors_comm2}.

\begin{propos}\label{propos_K_aux1} Let $i$, $j$, and~$k$ be three distinct indices from~$\{1,\ldots,n\}$.
Suppose that $f\in\IA_n$ \textnormal{(}or $f\in\IO_n$\textnormal{)} is an automorphism such that $\bigl[(f,K_{ik})\bigr]$ is an $S_{\fX}$-torsion element. Then  $\bigl[(f,K_{jk})\bigr]$ is also an $S_{\fX}$-torsion element.
\end{propos}

\begin{propos}\label{propos_K_double}
For any pair of double indexed Magnus generators~$K_{ij}$ and~$K_{kl}$, the corresponding element
$\bigl[(K_{ij},K_{kl})\bigr]$ is an $S_{\fX}$-torsion element.
\end{propos}

\begin{proof}
 If $i\notin \{k,l\}$ and $k\notin \{i,j\}$, then $K_{ij}$ and~$K_{kl}$ commute, so the required assertion is trivial. Hence, we suffice to consider the commutators of the form~$(K_{ij},K_{ik})$, $(K_{ij},K_{jk})$, and $(K_{ij},K_{ji})$, where the subscripts denoted by different letters are assumed to be different.

 First, consider the commutator~$(K_{ij},K_{ik})$. It acts on~$F_n$ as follows:
 \begin{align*}
 (K_{ij},K_{ik})&\colon
 \left\{
 \begin{aligned}
  a_i&\mapsto (a_j,a_k)a_i(a_j,a_k)^{-1},\\
  a_r&\mapsto a_r,\quad\qquad r\ne i.
 \end{aligned}
 \right.
 \end{align*}
Therefore, $(K_{ij},K_{ik})$ commutes with~$K_{ijk}$. Since $\tau(K_{ijk})\in\fX$, it follows from Lemma~\ref{lem_tors_comm1} that $\bigl[(K_{ij},K_{ik})\bigr]$ is an $S_{\fX}$-torsion element.

Second, since we already know that $\bigl[(K_{ij},K_{ik})\bigr]$ is an $S_{\fX}$-torsion element, Proposition~\ref{propos_K_aux1} implies that $\bigl[(K_{ij},K_{jk})\bigr]$ is an $S_{\fX}$-torsion element, too.

Finally, swapping the elements in the last commutator, we see that $\bigl[(K_{jk},K_{ij})\bigr]$ is an $S_{\fX}$-torsion element. Applying once more Proposition~\ref{propos_K_aux1}, we obtain that $\bigl[(K_{jk},K_{kj})\bigr]$ is also an $S_{\fX}$-torsion element.
\end{proof}

\begin{propos}\label{propos_K_aux2}
Suppose that $f\in\IA_n$ \textnormal{(}or~$f\in\IO_n$\textnormal{)} is an automorphism such that $\bigl[(f,K_{ij})\bigr]$ are $S_{\fX}$-torsion elements for all double indexed Magnus generators~$K_{ij}$. Then $\bigl[(f,K_{ijk})\bigr]$ are $S_{\fX}$-torsion elements for all triple indexed Magnus generators~$K_{ijk}$, too.
\end{propos}

\begin{proof}
 We suffice to prove that $\bigl[(f,K_{123})\bigr]$ is an $S_{\fX}$-torsion element.
 Set $$L=K_{13}K_{12}^{-1}K_{13}^{-1}K_{123}K_{12}.$$ Then $L$ acts on~$F_n$ as follows:
 \begin{equation*}
 L\colon
 \left\{
 \begin{aligned}
  a_1&\mapsto \bigl(a_3,a_2^{-1}\bigr)a_1,\\
  a_r&\mapsto a_r,\quad\qquad r\ne 1.
 \end{aligned}
 \right.
 \end{equation*}
 Therefore, $L$ commutes with~$K_{123}$, that is,
 \begin{equation}\label{eq_K_relation}
  K_{123}K_{13}K_{12}^{-1}K_{13}^{-1}K_{123}K_{12}=K_{13}K_{12}^{-1}K_{13}^{-1}K_{123}K_{12}K_{123}.
  \end{equation}
 Consider equation~\eqref{eq_tors_identity} associated with this relation and the element $x=f$. In this equation, all summands corresponding to double indexed letters~$K_{ij}^{\pm1}$ are $S_{\fX}$-torsion elements. Hence, the algebraic sum of the four summands corresponding to the four entries of~$K_{123}$ in~\eqref{eq_K_relation} must be an $S_{\fX}$-torsion element, too. Namely,
\begin{multline*}
\left(1+e_{\tau(K_{123})-\tau(K_{12})}-e_{-\tau(K_{12})}-e_{\tau(K_{123})}\right)\cdot\bigl[(f,K_{123})\bigr]\\{}=-\left(e_{\tau(K_{123})}-1\right)\left(e_{\tau(K_{12})}-1\right)e_{-\tau(K_{12})}\cdot\bigl[(f,K_{123})\bigr]
\end{multline*}
is an $S_{\fX}$-torsion element. Thus, $\bigl[(f,K_{123})\bigr]$ is also an $S_{\fX}$-torsion element.
\end{proof}

\begin{remark}
 Relation~\eqref{eq_K_relation} is one of the set of defining relations for the subgroup $\langle K_{12},K_{13},K_{123}\rangle$ of~$\IA_3$ found by Chein~\cite{Che69}. More precisely, he considered the elements
 \begin{align*}
 R_{pq}&=K_{12}^pK_{13}^{q-1}K_{12}^{-1}K_{123}K_{12}K_{13}^{1-q}K_{12}^{-p},\\
 L_{rs}&=K_{12}^rK_{13}^sK_{12}^{-1}K_{13}^{-1}K_{123}K_{12}K_{13}^{1-s}K_{12}^{-r},
 \end{align*}
 and proved that $R_{pq}$ multiplies~$a_1$ by a word in~$a_2$ and~$a_3$ from the right, while $L_{rs}$  multiplies~$a_1$ by a word in~$a_2$ and~$a_3$ from the left, thus, $R_{pq}$ commutes with~$L_{rs}$. One of these commutation relations, for $p=q=s=1$ and~$r=0$, is relation~\eqref{eq_K_relation}. Note that a simpler presentation for $\langle K_{12},K_{13},K_{123}\rangle$ was found by McCool~\cite{MCC88}.  Nevertheless, for our purposes, Chein's relations are more convenient.
\end{remark}

\begin{proof}[Proof of Proposition~\ref{propos_IA_torsion}]
 By Corollary~\ref{cor_tors_gen} we suffice to show that the classes in $[\IA_n,\IA_n]^{\ab}$ (or $[\IO_n,\IO_n]^{\ab}$) of the commutators of all pairs of the Magnus generators are $S_{\fX}$-torsion elements. By Proposition~\ref{propos_K_double}, the commutator of any pair of double indexed Magnus generators gives an $S_{\fX}$-torsion element. Then it follows from Proposition~\ref{propos_K_aux2} that a commutator of any double indexed generator and any triple indexed generator gives an $S_{\fX}$-torsion element. Finally, using Proposition~\ref{propos_K_aux2} once more, we obtain that the commutator of any pair of triple indexed Magnus generators also gives an $S_{\fX}$-torsion element.
\end{proof}

Combining Proposition~\ref{propos_IA_torsion} with Corollary~\ref{cor_IO_tau}, we obtain the following statement.

\begin{cor}\label{cor_tors_IO}
 If $n\ge 3$, then $[\IO_n,\IO_n]^{\ab}$ is an $S_{W_n\setminus\{0\}}$-torsion $\Z W_n$-module.
\end{cor}

Theorem~\ref{thm_main_OA} follows from Corollaries~\ref{cor_main_gen}, \ref{cor_SDP_IA}, and~\ref{cor_tors_IO}.

\begin{remark}\label{remark_IA}
As we have already mentioned, our apporach does not work in the case of the group~$[\IA_3,\IA_3]^{\ab}$. Note, however, that our method still allows to prove the finite generation of the groups~$[\IA_n,\IA_n]^{\ab}$ for $n\ge 4$, though this is certainly not a new result. Indeed, the only reason why the above proof fails for~$[\IA_n,\IA_n]^{\ab}$ is that the elements~$\tau(K_{ijk})$ lie in the kernel of the contraction homomorphism~$C$, see Corollary~\ref{cor_IA_tau}. So the proof will work literally in the same way as for~$[\IO_n,\IO_n]^{\ab}$ if we manage to improve
Proposition~\ref{propos_IA_torsion} in the following way.
\end{remark}

\begin{propos}\label{propos_IA_torsion_improve}
Suppose  that $n\ge 4$. If $\fX$ is an arbitrary subset of~$\widetilde{W}_n\setminus\{0\}$ that contains all elements~$\tau(K_{ij})$, then $[\IA_n,\IA_n]^{\ab}$ is an $S_{\fX}$-torsion $\Z \widetilde{W}_n$-module.
\end{propos}

In the remainder of this section we prove this proposition. The author does not know whether the assertion of Proposition~\ref{propos_IA_torsion_improve} remains true for $n=3$. If yes, then the finite generation of~$[\IA_3,\IA_3]^{\ab}$ would follow.

To prove Proposition~\ref{propos_IA_torsion_improve} we need to repeat the proof of Proposition~\ref{propos_IA_torsion}, modifying all the steps where we used the fact that $\tau(K_{ijk})\in\fX$. We used this fact twice: first, to prove that $\bigl[(K_{ij},K_{ik})\bigr]$ is an $S_{\fX}$-torsion element, and second, to prove Proposition~\ref{propos_K_aux2}. Let us perform this steps (for $n\ge 4$) without using of $\tau(K_{ijk})\in\fX$.

\begin{proof}[\protect{Proof that $\bigl[(K_{ij},K_{ik})\bigr]$ is an $S_{\fX}$-torsion element}]

Let $l$ be an element of $\{1,\ldots,n\}$ different from $i$, $j$, and~$k$. Then $K_{lk}$ commutes with both~$K_{ij}$ and~$K_{ik}$ and~$\tau(K_{lk})\in\fX$. So by Lemma~\ref{lem_tors_comm1} we see that $\bigl[(K_{ij},K_{ik})\bigr]$ is an $S_{\fX}$-torsion element.
\end{proof}

\begin{proof}[Proof of Proposition~\ref{propos_K_aux2} for $n\ge 4$ without using of $\tau(K_{ijk})\in\fX$]
Let $l$ be an element of $\{1,\ldots,n\}$ different from $i$, $j$, and~$k$. Then $K_{lk}$ commutes with~$K_{ijk}$ and~$\tau(K_{lk})\in\fX$. Since we know that $\bigl[(f,K_{lk})\bigr]$ is an $S_{\fX}$-torsion element, Lemma~\ref{lem_tors_comm2} yields that  $\bigl[(f,K_{ijk})\bigr]$ is an $S_{\fX}$-torsion element.
\end{proof}

Proposition~\ref{propos_IA_torsion_improve} follows. Thus, $[\IA_n,\IA_n]^{\ab}$ is finitely generated for $n\ge4$.

\begin{remark}\label{remark_Chein}
As we already noted in the Introduction, the group $[\IA_3,\IA_3]^{\ab}$ is apparently not finitely generated.  Below we provide additional evidence in support of this. Let~$T$ be the subgroup of~$\IA_3$ generated by the three triple indexed Magnus generators~$K_{123}$, $K_{231}$, and~$K_{312}$. Bachmuth~\cite{Bac65} proved that this subgroup is free, $T\cong F_3$. Now, let $N$ be the normal closure in~$\IA_3$ of the subgroup generated by the six double indexed Magnus generators~$K_{ij}$. Chein~\cite{Che69} conjectured that $T\cap N=\{1\}$. Though still unproven, this conjecture is highly plausible, as no element in the intersection has been found so far. The truth of Chein's conjecture would imply $T\cong \IA_3/N$, whence $\IA_3$ admits a quotient isomorphic to~$F_3$. It would then follow that the commutator subgroup of~$\IA_3$ admits a quotient isomorphic to~$F_{\infty}$, and therefore~$[\IA_3,\IA_3]^{\ab}$ cannot be finitely generated.
\end{remark}

\section{Some functional equations on abelian groups}\label{section_equations}

This section is auxiliary for the proof of Theorem~\ref{thm_explicit}. We introduce the notions of polynomial and quasipolynomial with values in an abelian group, and then study four specific functional equations, which then will be used in Section~\ref{section_quasi}.

\subsection{Polynomials and quasipolynomials}
Fix an abelian group~$B$. We will consider functions on~$\Z$ with values in~$B$. Introduce the \textit{shift operator}~$T$ and the \textit{difference operator}~$\Delta$ on such functions by
\begin{align*}
(Tf)(t)&=f(t+1),\\
(\Delta f)(t)&=f(t+1)-f(t).
\end{align*}
Then $\Delta=T-1$, where $1$ denotes the identity operator. We will conveniently use notation
$$
f'=\Delta f.
$$
The following definition is standard.

\begin{defin}
 A function $f\colon\Z\to B$ is said to be a \textit{polynomial of degree~$\le k$} if $\Delta^{k+1}f=0$. Moreover, $f$ is called a \textit{polynomial of degree~$k$} if it is a polynomial of degree~$\le k$ but not a polynomial of degree~$\le k-1$.
\end{defin}

It is an easy standard fact that any polynomial~$f$  of degree~$k$ can be written in a unique way as
$$
f(t)=b_0+tb_1+\binom{t}{2}b_2+\cdots+\binom{t}{k}b_k,\qquad b_i\in B.
$$
An important role in our considerations will be played by a somewhat larger class of functions. We conveniently give the following definition.

\begin{defin}
 A function $f\colon\Z\to B$ is said to be a \textit{quasipolynomial of degree~$\le k$} if $(T+1)\Delta^{k+1}f=0$. Moreover, $f$ is called a \textit{quasipolynomial of degree~$k$} if it is a quasipolynomial of degree~$\le k$ but not a quasipolynomial of degree~$\le k-1$.
\end{defin}

Note that quasipolynomials of degree~$0$ are exactly functions that take exactly two values, one on all even numbers and another on all odd numbers. An important example of a quasipolynomial $\Z\to\Z$ of degree~$1$ is the function
$$
n\mapsto \left\lfloor\frac{n}{2}\right\rfloor=\frac{n}{2}-\frac{1-(-1)^n}4\,.
$$
It follows immediately from the definition that $f$ is a quasipolynomial of degree $k>0$ if and only if $f'$ is a quasipolynomial of degree~$k-1$.

We conveniently denote by~$\CQ_k$ the abelian group of all quasipolynomials on~$\Z$ of degree~$\le k$ and write
$$
f(t)\equiv g(t)\pmod{\CQ_k}
$$
to indicate that the difference $f(t)-g(t)$ is a quasipolynomial of degree~$\le k$.

Now, consider functions $f\colon A\to B$, where $A$ is an arbitrary free abelian group. For each element $h\in A$, let~$\Delta_h$ be the corresponding difference operator given by
$$
(\Delta_hf)(a)=f(a+h)-f(a).
$$
If $e_1,\ldots,e_r$ is a basis of~$A$, then we use notation $\Delta_{(n_1,\ldots,n_r)}=\Delta_{n_1e_1+\cdots+n_re_r}$. Moreover, if $t_1,\ldots,t_r$ are the coordinates in~$A$ corresponding to the basis $e_1,\ldots,e_r$, then  we write $\Delta_{t_i}=\Delta_{e_i}$ for the partial difference operator with respect to~$t_i$.

There are several (generically non-equivalent) definitions of polynomials $f\colon A\to B$, see~\cite{Lac04}. We will not need them. Henceforth, the abelian group~$A$ will usually be identified with~$\Z^r$ and the first coordinate in~$\Z^r$ (usually denoted by~$t$) will play a distinguished role. Namely, we will consider functions $f\colon A\to B$ as functions in~$t$ depending on other coordinates~$\bq$ treated as parameters. For such functions, we will be interested in the properties of being polynomial or quasipolynomial in~$t$ for every fixed~$\bq$. We write
$$
f'(t,\bq)=(\Delta_tf)(t,\bq)=f(t+1,\bq)-f(t,\bq).
$$
Moreover, we write
$$
f(t,\bq)\equiv g(t,\bq)\pmod{\CQ_k}
$$
to indicate that the difference $f(t,\bq)-g(t,\bq)$ is a quasipolynomial of degree~$\le k$ in~$t$ for every fixed~$\bq$.

In the remainder of this section under a function we always mean a function with values in a fixed abelian group~$B$.

\subsection{First equation}

\begin{propos}\label{propos_func1}
 Let $f$ be a function on~$\Z^2$. Suppose that $f$ satisfies the functional equation
 \begin{equation}\label{eq_func1}
 f(k+1,l)+f(k,l+1)+f(k-1,l-1)=f(k-1,l)+f(k,l-1)+f(k+1,l+1).
 \end{equation}
 Then there exist functions $\kappa$, $\lambda$, and~$\mu$ on~$\Z$ such that
 \begin{equation}\label{eq_func1_sol}
 f(k,l)=\kappa(k)+\lambda(l)+\mu(k-l).
 \end{equation}
\end{propos}

To prove this proposition, we need the following auxiliary lemma.

\begin{lem}\label{lem_func1}
 Suppose that $f$ satisfies~\eqref{eq_func1} and in addition satisfies
 \begin{align}
  f(k,0)=f(k,1)&=0\ \ \ \text{for all}\ k,\label{eq_func1_cond1}\\
  f(0,l)&=0\ \ \ \text{for all}\ l.\label{eq_func1_cond2}
 \end{align}
 Then $f(k,l)=0$ for all~$k$ and~$l$.
\end{lem}

\begin{proof}
 Let us prove that $f(k,l)=0$ for all~$k$ and all~$l\ge 0$ by induction on~$l$. The base case for $l=0$ and~$l=1$ is true. Assume that we already know that $f(k,l-1)=0$ and $f(k,l)=0$ for all~$k$. Then equation~\eqref{eq_func1} reads as
 $$
 f(k,l+1)=f(k+1,l+1).
 $$
 Hence,
 $$
 f(k,l+1)=f(0,l+1)=0
 $$
 for all~$k$, which completes the induction step. Similarly, we can prove that $f(k,l)=0$ for all~$k$ and all~$l<0$ by descending induction on~$l$.
\end{proof}

\begin{proof}[Proof of Proposition~\ref{propos_func1}]
Let us find functions~$\kappa$, $\lambda$ and~$\mu$ such that equality~\eqref{eq_func1_sol} holds whenever either $k=0$ or~$l\in\{0,1\}$. Then the function
$$
\tilde{f}(k,l)=f(k,l)-\kappa(k)-\lambda(l)-\mu(k-l)
$$
would satisfy equation~\eqref{eq_func1} and conditions~\eqref{eq_func1_cond1} and~\eqref{eq_func1_cond2}. Thus, by Lemma~\ref{lem_func1} we would obtain that~\eqref{eq_func1_sol} holds for all~$k$ and~$l$.

We take $\kappa(0)=0$ and $\lambda(0)=\lambda(1)=0$. Then equations~\eqref{eq_func1_sol} for the pairs~$(k,0)$, $(k,1)$, and~$(0,l)$ read as
\begin{align}
 \kappa(k)+\mu(k)&=f(k,0),\label{eq_ag1}\\
 \kappa(k)+\mu(k-1)&=f(k,1),\label{eq_ag2}\\
 \lambda(l)+\mu(-l)&=f(0,l),\label{eq_bg}
\end{align}
respectively. Note that we only need to consider equations~\eqref{eq_bg} for $l\ne 0,1$, since for $l=0$ and $l=1$ these equations are already present among equations~\eqref{eq_ag1} and~\eqref{eq_ag2}.

Expressing~$\mu(k)$ from~\eqref{eq_ag1} and substituting the result to~\eqref{eq_ag2}, we obtain that
\begin{equation*}
 \kappa(k)-\kappa(k-1)=f(k,1)-f(k-1,0).
\end{equation*}
Obviously, this difference equation has a unique solution with $\kappa(0)=0$. Then we put
\begin{gather*}
 \mu(k)=f(k,0)-\kappa(k),\\
 \lambda(l)=f(0,l)-\mu(-l)=f(0,l)-f(-l,0)+\kappa(-l),\qquad l\ne 0,1.
\end{gather*}
Thus, we have constructed a solution of the system of equations~\eqref{eq_ag1}--\eqref{eq_bg}, which completes the proof of the proposition.
\end{proof}

\subsection{Second equation}

\begin{propos}\label{propos_6}
 Let $f_i$ be functions defined on~$\Z$, where $i=1,\ldots,6$. Suppose that $f_i$ satisfy the functional equation
 \begin{equation}\label{eq_6}
 \begin{split}
   f_1(k)+f_2(l)+f_3(m)+f_4(k-l)+f_5(k-m)+f_6(k-l-m)=0.
  \end{split}
 \end{equation}
 Then all functions~$f_i$ are polynomials of degree~$\le 1$, that is,
 $$
 f_i(k)=ka_i+b_i,\qquad a_i,b_i\in B.
 $$
\end{propos}
\begin{proof}
Consider the partial difference operators~$\Delta_k$, $\Delta_l$, and~$\Delta_m$.

First, we apply to equation~\eqref{eq_6} the operator~$\Delta_l\Delta_m$. We obtain that $f_6''=0$ and hence $f_6$ is a polynomial of degree~$\le 1$.

Second, we apply to~\eqref{eq_6} the operator~$\Delta_k\Delta_l$. Since we already know that $f_6''=0$, we obtain that $f_4''=0$ and hence $f_4$ is a polynomial of degree~$\le 1$. Similarly, $f_5$ is a polynomial of degree~$\le 1$.

Third, we apply to~\eqref{eq_6} the operator~$\Delta_l^2$. Since we already know that $f_4''=0$ and~$f_6''=0$, we obtain that $f_2''=0$ and hence $f_2$ is a polynomial of degree~$\le 1$. Similarly, $f_3$ is a polynomial of degree~$\le 1$.

Finally, from~\eqref{eq_6} it follows that $f_1$ is also a polynomial of degree~$\le 1$.
\end{proof}

\subsection{Third equation}

\begin{propos}\label{propos_9}
 Let $f_i$ be functions defined on~$\Z$, where $i=1,\ldots,10$. Suppose that $f_i$ satisfy the functional equation
 \begin{equation}\label{eq_9}
 \begin{split}
   f_1(t+m)&+f_2(t-k+m)+f_3(t-l+m)+f_4(t)\\{}&+f_5(t-k)+f_6(t-l)+f_7(t-k-l)\\{}&+f_8(t-k-m)+f_9(t-l-m)+f_{10}(t-k-l-m)=0.
  \end{split}
 \end{equation}
 Then all functions~$f_i$ are quasipolynomials of degree~$\le 2$. Moreover, $f_1$, $f_4$, $f_7$, and~$f_{10}$ are polynomials of degree~$\le 2$.
\end{propos}

\begin{proof}
Consider the abelian group~$\Z^4$ with coordinates~$(t,k,l,m)$ and the corresponding difference operators $\Delta_{(a,b,c,d)}$ on functions on~$\Z^4$.

First, applying to~\eqref{eq_9} the operator~$\Delta_{(0,1,0,0)}\Delta_{(0,0,1,0)}=\Delta_k\Delta_l$, we obtain that
$$
f_7''(t-k-l)+f_{10}''(t-k-l-m)=0.
$$
Similarly, applying to~\eqref{eq_9} the operator~$\Delta_{(1,1,0,0)}\Delta_{(1,0,1,0)}$, we obtain that
$$
f_1''(t+m)+f_4''(t)=0.
$$
It follows that $f_1''(t)$, $f_4''(t)$, $f_7''(t)$, and~$f_{10}''(t)$ are constants. Hence, $f_1$, $f_4$, $f_7$, and~$f_{10}$ are polynomials of degree~$\le 2$.

Second, apply to~\eqref{eq_9} the operator~$\Delta_{(0,0,0,1)}=\Delta_m$. We get the equation
\begin{multline}\label{eq_1238910}
f_1'(t+m)+f_2'(t-k+m)+f_3'(t-l+m)-f_8'(t-k-m-1)\\{}-f_9'(t-l-m-1)-f_{10}'(t-k-l-m-1)=0.
\end{multline}
Here we use that $\Delta_m\bigl(f(s-m)\bigr)=-f'(s-m-1)$. Applying the operator~$\Delta_{(0,-1,-1,1)}$ to~\eqref{eq_1238910} and taking into account that we already know that $f_1''(t)$ and~$f_{10}''(t)$ are constants, we obtain that
$$
f_2'(t-k+m+2)-f_2'(t-k+m)+f_3'(t-l+m+2)-f_3'(t-l+m)=\mathrm{const}.
$$
Hence $f_2'(t+2)-f_2'(t)$ and $f_3'(t+2)-f_3'(t)$ are constants, which immediately implies that $f_2$ and~$f_3$ are quasipolynomials of degree~$\le 2$. Similarly, applying to~\eqref{eq_1238910} the operator~$\Delta_{(0,1,1,1)}$, we will obtain that $f_8$ and~$f_9$ are quasipolynomials of degree~$\le 2$.

Finally, substitute $k=0$, $l=t$, and~$m=0$ to~\eqref{eq_9}. We obtain the equation
\begin{equation*}
 \begin{split}
   f_1(t)&+f_2(t)+f_3(0)+f_4(t)+f_5(t)+f_6(0)+f_7(0)+f_8(t)+f_9(0)+f_{10}(0)=0.
  \end{split}
 \end{equation*}
Since we already know that $f_1$, $f_2$, $f_4$, and~$f_8$ are quasipolynomials of degree~$\le 2$, we obtain that $f_5$ is a quasipolynomial of degree~$\le2$, too. Similarly, $f_6$ is also a quasipolynomial of degree~$\le2$.
\end{proof}

\subsection{Fourth equation}

\begin{propos}\label{propos_func2}
 Let $f_i$ be functions defined on~$\Z^2$, where $i=1,\ldots,4$. Suppose that $f_i$ satisfy the functional equation
 \begin{equation}\label{eq_func2}
  f_1(k,l)+f_2(k-m,l)=f_3(k,m)+f_4(k-l,m).
 \end{equation}
 Then there exist functions $g_1,\ldots,g_4$, and~$h$ defined on~$\Z$ such that
 \begin{align}
  f_1(k,l)&=g_1(k)+g_2(k-l)+h(l),\label{eq_f1}\\
  f_2(k,l)&=g_3(k)+g_4(k-l)-h(l).\label{eq_f2}
 \end{align}
\end{propos}

\begin{proof}
Let us show that each of the functions~$f_i$ satisfies functional equation~\eqref{eq_func1}. Consider the abelian group~$\Z^3$ with coordinates~$(k,l,m)$ and the corresponding difference operators $\Delta_{(a,b,c)}$ on functions on~$\Z^3$. We put
\begin{align*}
 D_1&=\Delta_{(0,0,1)},&D_2&=\Delta_{(1,0,1)},&D_3&=\Delta_{(0,1,0)},&D_4&=\Delta_{(1,1,0)}.
\end{align*}
Then
\begin{align*}
 D_1\bigl(f_1(k,l)\bigr)&=0,&D_2\bigl(f_2(k-m,l)\bigr)&=0,\\ D_3\bigl(f_3(k,m)\bigr)&=0,&D_4\bigl(f_4(k-l,m)\bigr)&=0.
\end{align*}
Applying the operators~$D_2D_3D_4$, $D_1D_3D_4$, $D_1D_2D_4$, and~$D_1D_2D_3$ to both sides of equation~\eqref{eq_func2}, we obtain the functional equations~\eqref{eq_func1} for the functions~$f_1$, $f_2$, $f_3$, and~$f_4$, respectively. By Proposition~\ref{propos_func1} we have
$$
f_i(k,l)=\kappa_i(k)+\lambda_i(l)+\mu_i(k-l)
$$
for  some functions~$\kappa_i$, $\lambda_i$, and~$\mu_i$. Equation~\eqref{eq_func2} now reads  as
\begin{multline}
 \bigl(\kappa_1(k)-\kappa_3(k)\bigr)+\bigl(\lambda_1(l)+\lambda_2(l)\bigr)-\bigl(\lambda_3(m)+\lambda_4(m)\bigr)+\bigl(\mu_1(k-l)-\kappa_4(k-l)\bigr)\\{}+\bigl(\kappa_2(k-m)-\mu_3(k-m)\bigr)+\bigr(\mu_2(k-l-m)-\mu_4(k-l-m)\bigr)=0.
 \end{multline}
Note that this is an equation of the form~\eqref{eq_6}.
By Proposition~\ref{propos_6} we have that
$$
\lambda_1(l)+\lambda_2(l)=la+b
$$
for some elements~$a,b\in B$. Now, we put
\begin{gather*}
g_1(k)=\kappa_1(k),\qquad g_2(k)=\mu_1(k),\qquad h(k)=\lambda_1(k),\\ g_3(k)=\kappa_2(k)+ka+b,\qquad g_4(k)=\mu_2(k)-ka,
\end{gather*}
and obtain formulas~\eqref{eq_f1} and~\eqref{eq_f2}.
\end{proof}

\section{Configuration of curves for the proof of Theorem~\ref{thm_explicit}}\label{section_config}
Recall that $\gamma$ is a separating simple closed curve on~$\Sigma_g^b$, where $g\ge 3$ and $b\in\{0,1\}$, such that $\gamma$ is not homotopic to either a point or a boundary component, and $H=P_1\oplus P_2$ is the corresponding splitting in the homology of~$\Sigma_g^b$. Consider a rank~$3$ subgroup $L\subset H$  that satisfy all conditions from Theorem~\ref{thm_explicit}. By swapping~$P_1$ and~$P_2$, we can achieve that $\rk (L\cap P_1)=1$ and~$\rk (L\cap P_2)=2$. From the conditions on the subgroup~$L$ it follows that $L\cap P_1$ and~$L\cap P_2$ are isotropic direct summands of~$P_1$ and~$P_2$, respectively. Choose a generator~$a_1$ of~$L\cap P_1$ and a basis~$a_2,a_3$ of~$L\cap P_2$.
Changing sign of~$a_1$, we can achieve that
$$
u_L=a_1\wedge a_2\wedge a_3.
$$
Now, choose oriented simple closed curves $\alpha_1$, $\alpha_1'$, $\alpha_2$, $\alpha_2'$, $\alpha_3$, and~$\alpha_3'$ such that
\begin{itemize}
 \item these curves are disjoint from~$\gamma$ and from each other,
 \item for each~$i$, we have $[\alpha_i]=[\alpha_i']=a_i$,
 \item the connected component~$S$ of $\Sigma_g^b\setminus\left(\bigcup_{i=1}^3\left(\alpha_i\cup\alpha_i'\right)\right)$ that contains~$\gamma$ is homeomorphic to a sphere with $6$ punctures,
 \item $S$ adjoins the the curves~$\alpha_1$, $\alpha_2$, and~$\alpha_3$ along their right sides and the curves~$\alpha_1'$, $\alpha_2'$, and~$\alpha_3'$ along their left sides.
\end{itemize}

Let us position the surface~$\Sigma_g^b$ and the curves~$\gamma$, $\alpha_i$, and~$\alpha_i'$ on it as in Fig.~\ref{fig_S}. This figure depicts a sphere with three handles. If $g>3$, additional handles should be attached within the three regions colored in green. Moreover, a boundary component (if present) should also be contained within one of these regions. Note that, in the case $g=3$, $b=0$, the curves $\alpha_i'$ are isotopic to the curves~$\alpha_i$, respectively.
In what follows, all our constructions will be contained in the closure~$\oS$ of~$S$. So we draw it separately and consider additional curves~$\xi_i$, $\eta_i$, and~$\zeta_i$ on it, see Fig.~\ref{fig_main}. (Note that $\oS$ is a sphere with $6$ boundary components. So the outermost line in Fig.~\ref{fig_main} is not a boundary component but merely the contour of the visible part of the sphere.) We introduce notation:
\begin{equation*}
 x_i=T_{\xi_i},\qquad
 y_i=T_{\eta_i},\qquad
 z_i=T_{\zeta_i}.
\end{equation*}
Then
\begin{itemize}
 \item the mapping classes~$y_1$, $y_2$, and~$y_3$ pairwise commute,
 \item the mapping classes~$z_1$, $z_2$, and~$z_3$ pairwise commute,
 \item $z_iy_i=y_iz_i$,
 \item $x_iy_{i+1}=y_{i+1}x_i$,
 \item $x_iz_{i-1}=z_{i-1}x_i$,
\end{itemize}
where the indices are considered modulo~$3$.

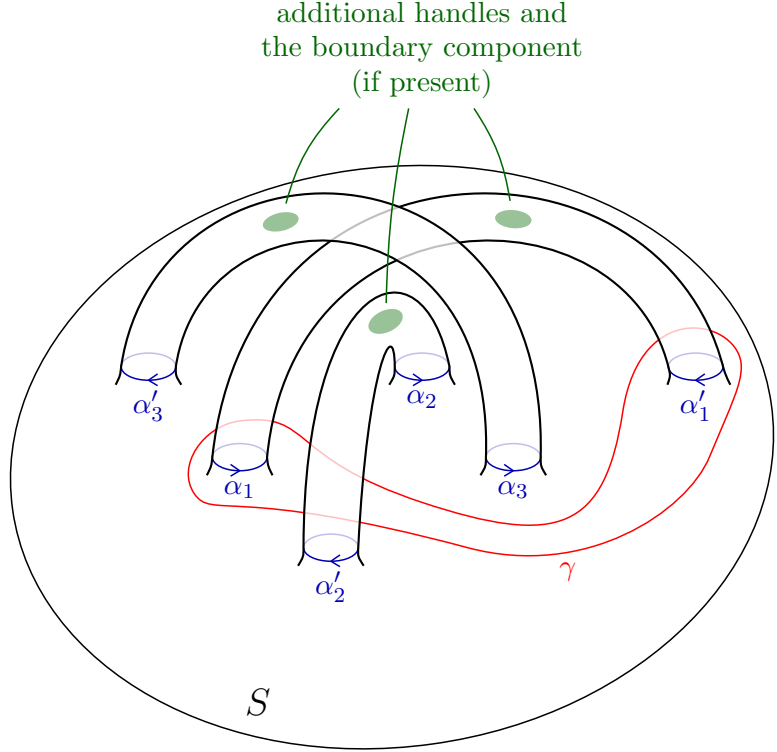
\begin{figure}[t]
  \begin{center}
   \begin{tikzpicture}[scale=1.2,
        mid arrow/.style={
            postaction={
                decorate,
                decoration={
                    markings,
                    mark=at position 0.55 with {
                        \arrow[#1]{stealth}
                    }
                }
            }
        },
        mid arrow/.default=black,
    ]

\unitlength=1cm

\small
\definecolor{myblue}{rgb}{0, 0, 0.7}
\definecolor{mygreen}{rgb}{0, 0.4, 0}
\tikzset{every path/.append style={line width=.2mm}}
\tikzset{my dash/.style={dash pattern=on 2pt off 1.5pt}}

\begin{scope}[xshift = .67 cm, yshift = 1 cm]
\draw[rotate = 8] (0,0) ellipse (4.2 and 3.2);
\end{scope}

\draw [red] (4.5,2) .. controls (4.5,2.5) and (3.4,2.8) .. (3.2,1.5) .. controls (3,.2) and (2.5,0) .. (1,0.5) .. controls (-.5,1) and (-.2,1.5) .. (-1,1.4) .. controls (-1.4,1.35) and (-1.7,.9) .. (-1.5,.6) .. controls (-1.3,.3) and  (-1.1,.7) .. (1.5,0) .. controls (2.8,-.35) and (3.9,.4) .. (4.15,1) node [pos=.3,below] {$\gamma$}.. controls (4.4,1.6) and (4.5,1.8) .. (4.5,2);

\draw [myblue] (0,0) ellipse (.3 and .15);
\draw [myblue] (0,-.15) -- (.1,-.08);
\draw [myblue] (0,-.15) -- (.1,-.21);
\draw [myblue] (0,-.15) node [below]  {$\alpha_2'$};

\begin{scope}[xshift = -2cm, yshift = 2cm]
 \draw [myblue] (0,0) ellipse (.3 and .15);
\draw [myblue] (0,-.15) -- (.1,-.08);
\draw [myblue] (0,-.15) -- (.1,-.21);
\draw [myblue] (0,-.15) node [below]  {$\alpha_3'$};
\end{scope}

\begin{scope}[xshift = 4cm, yshift = 2cm]
 \draw [myblue] (0,0) ellipse (.3 and .15);
\draw [myblue] (0,-.15) -- (.1,-.08);
\draw [myblue] (0,-.15) -- (.1,-.21);
\draw [myblue] (0,-.15) node [below]  {$\alpha_1'$};
\end{scope}

\begin{scope}[xshift = -1cm, yshift = 1cm]
\draw [myblue] (0,0) ellipse (.3 and .15);
\draw [myblue] (0,-.15) -- (-.1,-.08);
\draw [myblue] (0,-.15) -- (-.1,-.21);
\draw [myblue] (0,-.15) node [below]  {$\alpha_1$};
\end{scope}

\begin{scope}[xshift = 2cm, yshift = 1cm]
\draw [myblue] (0,0) ellipse (.3 and .15);
\draw [myblue] (0,-.15) -- (-.1,-.08);
\draw [myblue] (0,-.15) -- (-.1,-.21);
\draw [myblue] (0,-.15) node [below]  {$\alpha_3$};
\end{scope}

\begin{scope}[xshift = 1cm, yshift = 2cm]
\draw [myblue] (0,0) ellipse (.3 and .15);
\draw [myblue] (0,-.15) -- (-.1,-.08);
\draw [myblue] (0,-.15) -- (-.1,-.21);
\draw [myblue] (0,-.15) node [below]  {$\alpha_2$};
\end{scope}


\fill[white, nearly opaque] (-.3,0) .. controls (-.3,3) and (1,3.5) .. (1.3,2) --
(.7,2) .. controls (.7,2.6) and (.4,2) .. (.3,0)--(-.3,0);

\draw [thick] (-.36,-.2).. controls (-.3,-.1) .. (-.3,0) .. controls (-.3,3) and (1,3.5) .. (1.3,2) .. controls (1.302,1.9) .. (1.36,1.8);
\draw [thick] (.36,-.2).. controls (.295,-.1) .. (.3,0) .. controls (.4,2) and (.7,2.6) ..
(.7,2) .. controls (.7,1.9) .. (.64,1.8);

\fill[white, nearly opaque] (-1.3,1) .. controls (-.9,4.7) and (3.6,4.7) .. (4.3,2) --
(3.7,2) .. controls (3.2,4) and (-.4,4) .. (-.7,1) -- (-1.3,1);

\draw [thick] (-1.36,.8) .. controls (-1.31,.9) ..  (-1.3,1) .. controls (-.9,4.7) and (3.6,4.7) .. (4.3,2) .. controls (4.33,1.9) .. (4.38,1.8);
\draw [thick] (-.64,.8) .. controls (-.69,.9) .. (-.7,1) .. controls (-.4,4) and (3.2,4) .. (3.7,2) .. controls (3.725,1.9) .. (3.69,1.8);

\fill[white, nearly opaque] (-2.3,2) .. controls (-1.9,4.7) and (2.4,4.7) .. (2.3,1) --
(1.7,1) .. controls (1.8,4) and (-1.4,4) .. (-1.7,2) -- (-2.3,2);

\draw [thick] (-2.36,1.8) .. controls (-2.31,1.9) .. (-2.3,2) .. controls (-1.9,4.7) and (2.4,4.7) .. (2.3,1) .. controls (2.3,0.9) .. (2.36,.8);
\draw [thick] (-1.64,1.8) .. controls (-1.703,1.9) .. (-1.7,2) .. controls (-1.4,4) and (1.8,4) .. (1.7,1) .. controls (1.7,.9) .. (1.64,.8);

%
%
%

%

\begin{scope}[xshift = 2cm, yshift =3.63cm]
\fill [mygreen!40, rotate=-5] (0,0) ellipse (.2 and .1) node [above=-1pt] (A) {};
\end{scope}

\begin{scope}[xshift = -.55cm, yshift =3.6cm]
\fill [mygreen!40, rotate=13] (0,0) ellipse (.2 and .1)node [above=-1pt] (B) {};
\end{scope}


\begin{scope}[xshift = .6cm, yshift =2.5cm]
\fill [mygreen!40, rotate=25] (0,0) ellipse (.2 and .12) node [above=-1pt] (C) {};
\end{scope}


\draw [mygreen] (1,5.5) node (D) {
\begin{minipage}{6cm}
\begin{center}
 additional handles and\\ the boundary component\\ (if present)
 \end{center}
\end{minipage}};

\draw [mygreen] (D) edge[bend left=12] (A);
\draw [mygreen] (D) edge[bend right=14] (B);
\draw [mygreen] (D) edge[bend right=5] (C);

\draw (-.8,-1.7) node {\large$S$};

\end{tikzpicture}
  \end{center}

 \caption{Surface~$\Sigma_g^b$ and curves~$\gamma$, $\alpha_i$, and~$\alpha_i'$}\label{fig_S}
\end{figure}

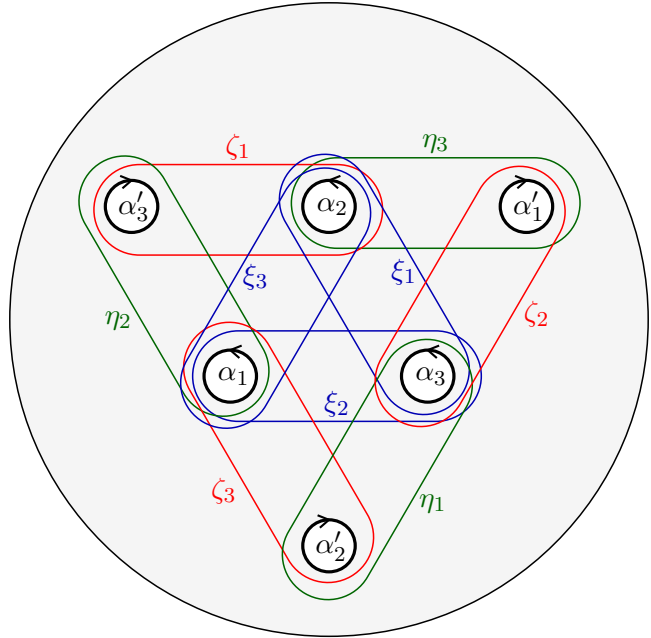
\begin{figure}[t]
 \begin{center}
  \begin{tikzpicture}[scale=1.5,
        mid arrow/.style={
            postaction={
                decorate,
                decoration={
                    markings,
                    mark=at position 0.55 with {
                        \arrow[#1]{stealth}
                    }
                }
            }
        },
        mid arrow/.default=black,
    ]

\unitlength=1cm

\small
\definecolor{myblue}{rgb}{0, 0, 0.7}
\definecolor{mygreen}{rgb}{0, 0.4, 0}
\tikzset{every path/.append style={line width=.2mm}}
\tikzset{my dash/.style={dash pattern=on 2pt off 1.5pt}}

\filldraw [fill=black!4] (0,0) circle (2.8);
\filldraw [fill=white, very thick] (0,1) circle (.23);
\filldraw [fill=white, very thick] (0.866,-0.5) circle (.23);
\filldraw [fill=white, very thick] (-0.866,-0.5) circle (.23);
\filldraw [fill=white, very thick] (0,-2) circle (.23);
\filldraw [fill=white, very thick] (1.732,1) circle (.23);
\filldraw [fill=white, very thick] (-1.732,1) circle (.23);
\draw [thick](0,1.23)-- (.1,1.27);
\draw [thick](0,1.23)-- (.08,1.16);
\begin{scope}[xshift=.866cm, yshift=-1.5cm]
 \draw [thick](0,1.23)-- (.1,1.27);
\draw [thick](0,1.23)-- (.08,1.16);
\end{scope}
\begin{scope}[xshift=-.866cm, yshift=-1.5cm]
 \draw [thick](0,1.23)-- (.1,1.27);
\draw [thick](0,1.23)-- (.08,1.16);
\end{scope}
\begin{scope}[yshift=-2cm]
 \draw [thick](0,.23)-- (-.1,.27);
\draw [thick](0,.23)-- (-.08,.16);
\end{scope}
\begin{scope}[xshift=1.732cm, yshift=1cm]
 \draw [thick](0,.23)-- (-.1,.27);
\draw [thick](0,.23)-- (-.08,.16);
\end{scope}
\begin{scope}[xshift=-1.732cm, yshift=1cm]
 \draw [thick](0,.23)-- (-.1,.27);
\draw [thick](0,.23)-- (-.08,.16);
\end{scope}

\begin{scope}
\begin{scope}
\draw [myblue] (-0.796,-0.9) arc (270:90:.4) -- (0.936,-.1) arc (90:-90:.4) -- (-0.796,-0.9)
node [pos=.5,above=-2pt] {$\xi_2$};
\end{scope}

\begin{scope}[xshift=.866cm, yshift=1.53cm]
\draw [ mygreen] (-0.796,-0.9) arc (270:90:.4) -- (0.936,-.1)
node [pos =.5, above =-2pt] {$\eta_3$} arc (90:-90:.4) -- (-0.796,-0.9);
\end{scope}

\begin{scope}[xshift=-.866cm, yshift=1.47cm]
\draw [red] (-0.796,-0.9) arc (270:90:.4) -- (0.936,-.1) node [pos =.5, above =-2pt] {$\zeta_1$} arc (90:-90:.4) -- (-0.796,-0.9);
\end{scope}
\end{scope}

\begin{scope}[rotate=120]
\begin{scope}
\draw [ myblue] (-0.796,-0.9) arc (270:90:.4) -- (0.936,-.1) arc (90:-90:.4) -- (-0.796,-0.9)
node [pos=.5,below left=-5pt] {$\xi_1$};
\end{scope}

\begin{scope}[xshift=.866cm, yshift=1.53cm]
\draw [ mygreen] (-0.796,-0.9) arc (270:90:.4) -- (0.936,-.1)
node [pos =.5, below left =-3pt] {$\eta_2$} arc (90:-90:.4) -- (-0.796,-0.9);
\end{scope}

\begin{scope}[xshift=-.866cm, yshift=1.47cm]
\draw [red] (-0.796,-0.9) arc (270:90:.4) -- (0.936,-.1)
node [pos =.5, below left =-3pt] {$\zeta_3$} arc (90:-90:.4) -- (-0.796,-0.9);
\end{scope}
\end{scope}

\begin{scope}[rotate=240]
\begin{scope}
\draw [myblue] (-0.796,-0.9) arc (270:90:.4) -- (0.936,-.1) arc (90:-90:.4) -- (-0.796,-0.9)
node [pos=.57,below right=-4pt] {$\xi_3$};
\end{scope}

\begin{scope}[xshift=.866cm, yshift=1.53cm]
\draw [mygreen] (-0.796,-0.9) arc (270:90:.4) -- (0.936,-.1)
node [pos =.5, below right =-3pt] {$\eta_1$} arc (90:-90:.4) -- (-0.796,-0.9);
\end{scope}

\begin{scope}[xshift=-.866cm, yshift=1.47cm]
\draw [red] (-0.796,-0.9) arc (270:90:.4) -- (0.936,-.1) node [pos =.38, below right =-3pt] {$\zeta_2$} arc (90:-90:.4) -- (-0.796,-0.9);
\end{scope}
\end{scope}

%

\path (0.03,1) node {$\alpha_2$};
\path (0.02,-2) node {$\alpha_2'$};
\path (-0.836,-0.5) node {$\alpha_1$};
\path (0.896,-0.5) node {$\alpha_3$};
\path (1.752,1) node {$\alpha_1'$};
\path (-1.712,1) node {$\alpha_3'$};

\end{tikzpicture}
 \end{center}

 \caption{Surface~$\oS\approx\Sigma_0^6$ and curves~$\xi_i$, $\eta_i$, and~$\zeta_i$}\label{fig_main}
 \end{figure}

 \begin{figure}[!t]
 \begin{center}
\unitlength=1mm
\tikzset{->-/.style={decoration={
     markings,
     mark=at position #1 with {\arrow{>}}},postaction={decorate}}}


\definecolor{mygreen}{rgb}{0, 0.4, 0}
\definecolor{myblue}{rgb}{0, 0, 0.7}
\definecolor{mypink}{rgb}{1, 0.3, 0.8}
\definecolor{myyellow}{rgb}{1,1, 0}

\colorlet{cola}{violet}
\colorlet{cold}{red}
\colorlet{colg}{myblue}
\colorlet{colkf}{mypink}
\colorlet{colks}{violet}
\colorlet{coll}{orange}
\colorlet{colx}{mygreen}
\colorlet{coly}{orange}
\colorlet{colm}{mygreen}
\colorlet{colz}{brown}

\begin{tikzpicture}
\small
\tikzset{every path/.append style={line width=.3mm}}

\filldraw [fill=black!4, very thick] (0,0) circle (1.9); 
\filldraw [fill=white, very thick] (0,1) circle (.3);
\filldraw [fill=white, very thick] (0.866,-0.5) circle (.3);
\filldraw [fill=white, very thick] (-0.866,-0.5) circle (.3);

\draw [thick] (0.1,-.5) ellipse (1.4 and .7);
\draw [thick, rotate=120] (0.1,-.5) ellipse (1.4 and .7);
\draw [thick, rotate=240] (0.1,-.5) ellipse (1.4 and .7);
\path (0.03,1) node {$b_2$};
\path (-0.836,-0.5) node {$b_1$};
\path (0.896,-0.5) node {$b_3$};
\path (0.13,-1.45) node {$c_2$};
\path (1.27,0.7) node {$c_1$};
\path (-1.27,0.7) node {$c_3$};
\path (1.56,-1.56) node {$b_4$};

\end{tikzpicture}
\end{center}
\caption{Lantern relation}\label{fig_lantern}
\vspace{5mm}

 \begin{center}
  \input{fig_lantern_three}
 \end{center}

 \caption{Three lantern configurations}\label{fig_delta}
\end{figure}

Recall the \textit{lantern relation} in $\Mod(\Sigma_0^4)$ discovered by Johnson~\cite{Joh79} (cf.~\cite{FaMa12}):
$$
T_{c_1}T_{c_2}T_{c_3}=T_{b_1}T_{b_2}T_{b_3}T_{b_4}
$$
for the curves $c_i$ and~$b_j$ shown in Fig.~\ref{fig_lantern}. Any $7$-tuple of curves
as in this figure will be referred to as a \textit{lantern configuration}.

\begin{propos}\label{propos_delta}
We have
 \begin{equation}\label{eq_delta}
 T_{\gamma}=z_3^{-1}y_2^{-1}x_3^{-1}x_2^{-1}
 T_{\alpha_1}^3T_{\alpha_1'}T_{\alpha_2}T_{\alpha_2'}T_{\alpha_3}T_{\alpha_3'}.
 \end{equation}
\end{propos}

\begin{proof}
Consider auxiliary simple closed curves~$\delta_1$, $\delta_2$, and~$\delta_3$ shown in Fig.~\ref{fig_delta}. Note that the curve~$\gamma$ in Fig.~\ref{fig_delta}(c) corresponds to the curve~$\gamma$ in Fig.~\ref{fig_S}. We have the following three relations in~$\Mod(\Sigma_g^b)$, since the corresponding $7$-tuples of simple closed curves form lantern configurations, see Fig.~\ref{fig_delta}, (a), (b), and~(c), respectively:
\begin{align}
 \label{eq_lant1}
 T_{\xi_1}T_{\xi_2}T_{\xi_3}&=T_{\alpha_1}T_{\alpha_2}T_{\alpha_3}T_{\delta_1},\\
 \label{eq_lant2}
 T_{\delta_1}T_{\eta_2}T_{\delta_2}&=T_{\alpha_1}T_{\alpha_3'}T_{\xi_1}T_{\delta_3},\\
 \label{eq_lant3}
 T_{\zeta_3}T_{\gamma}T_{\delta_3}&=T_{\alpha_1}T_{\alpha_1'}T_{\alpha_2'}T_{\delta_2}.
\end{align}
We first express~$T_{\delta_1}$ from~\eqref{eq_lant1} and substitute to~\eqref{eq_lant2} and then express~$T_{\delta_3}T_{\delta_2}^{-1}$ from the obtained relation and substitute to~\eqref{eq_lant3}. Taking into account that the twists~$T_{\alpha_i}$ and~$T_{\alpha_i'}$ commute with all Dehn twists under consideration, we get
$$
z_3T_{\gamma}x_2x_3y_2=T_{\alpha_1}^3T_{\alpha_1'}T_{\alpha_2}T_{\alpha_2'}T_{\alpha_3}T_{\alpha_3'},
$$
hence the required formula for~$T_{\gamma}$.
\end{proof}

The following proposition is a direct consequence of Proposition~\ref{propos_csip}.

\begin{propos}\label{propos_tau_zy}
 Suppose that $i\ne j$. Then the commutator $(z_i,y_j)$ belongs to~$\I_g^b$ and
 \begin{equation*}
  \tau\bigl((z_i,y_j)\bigr)=u_L.
 \end{equation*}
\end{propos}

Since the mapping classes~$y_i$ and~$z_i$ stabilize the element $u_L\in U^b_g$, Proposition~\ref{propos_tau_zy} implies the following statement.

\begin{cor}\label{cor_tau_zy}
 Suppose that $i,j,k$ is a permutation of~$1,2,3$ and $p,q\in\Z$. Then the commutator $\bigl(z_i,y_j^py_k^q\bigr)$ belongs to~$\I_g^b$ and
 $$\tau\left(\bigl(z_i,y_j^py^q_k\bigr)\right)=(p+q)u_L.$$
\end{cor}

\section{Functions~$F_i$ and functional equations for them}\label{section_sF}

Consider the free abelian group~$\Z^7$ with coordinates $t,m_1,m_2,m_3,n_1,n_2,n_3$. We will write $\bm=(m_1,m_2,m_3)$ and $\bn=(n_1,n_2,n_3)$; then a point of~$\Z^7$ has the form~$(t,\bm,\bn)$.
We introduce notation
$$
\bz^{\bm}=z_1^{m_1}z_2^{m_2}z_3^{m_3},\qquad \bfy^{\bn}=y_1^{n_1}y_2^{n_2}y_3^{n_3}.
$$
Since $z_1$, $z_2$, and~$z_3$ pairwise commute, we have $\bz^{\bm+\bm'}=\bz^{\bm}\bz^{\bm'}$, and similarly for~$\bfy$. Consider the commutator
$$
c=(z_1,y_3)=z_1y_3z_1^{-1}y_3^{-1}.
$$
By Proposition~\ref{propos_tau_zy} we have that $c\in\I_g^b$ and
\begin{equation}\label{eq_tau_c}
 \tau(c)=u_L.
\end{equation}
Also we put
\begin{align*}
\bnul&=(0,0,0),&\be_1&=(1,0,0),&\be_2&=(0,1,0),&\be_3&=(0,0,1).
\end{align*}

For each vector $(t,\bm,\bn)\in\Z^7$ and $i=1,2,3$, we consider the mapping class
\begin{align}\nonumber
s_i(t,\bm,\bn)&=c^t\bz^{\bm}\bfy^{\bn}z_i\bfy^{-\bn}\bz^{-\bm-\be_i}c^{n_{j}+n_{k}-t}\\
\label{eq_z}            {}&=c^t\bz^{\bm}\left(y_j^{n_j}y_k^{n_k},z_i\right)\bz^{-\bm}c^{n_{j}+n_{k}-t},
\end{align}
where $i,j,k$ is a permutation of~$1,2,3$. The two formulas for~$s_i(t,\bm,\bn)$ are equivalent, since $z_i$ and~$y_i$ commute.

\begin{propos}
 The mapping class~$s_i(t,\bm,\bn)$ belongs to~$\K_g^b$.
\end{propos}

\begin{proof}
 By Corollary~\ref{cor_tau_zy} we have that the commutator~$\left(y_j^{n_j}y_k^{n_k},z_i\right)$ belongs to~$\I_g^b$ and $$\tau\left(\left(y_j^{n_j}y_k^{n_k},z_i\right)\right)=-(n_j+n_k)u_L.$$
 Combining this with~\eqref{eq_tau_c} and using that the mapping classes~$z_l$ act trivially on~$u_L$, we obtain that $\tau\bigl(s_i(t,\bm,\bn)\bigr)=0$. Therefore, $s_i(t,\bm,\bn)\in\K_g^b$.
\end{proof}

We denote by~$F_i(t,\bm,\bn)$ the class of the element~$s_i(t,\bm,\bn)$ in~$(\K_{g}^b)^{\ab}$. Our approach consists in treating~$F_i$ as $(\K_{g}^b)^{\ab}$-valued functions on~$\Z^7$ and finding and studying functional equations for them. We first observe the following obvious consequence of~\eqref{eq_z}.

\begin{propos}\label{propos_independ}
 The mapping class $s_i(t,\bm,\bn)$, and hence the element~$F_i(t,\bm,\bn)$, does not depend on~$n_i$.
\end{propos}

The following proposition is a direct consequence of the fact that $z_i$ and~$z_j$ commute.

\begin{propos}
Suppose that $i,j,k$ is a permutation of~$1,2,3$. Then, for all $(t,\bm,\bn)\in\Z^7$, we have
\begin{equation}\label{eq_sisj}
s_i(t,\bm,\bn)s_j(t-n_j-n_k,\bm+\be_i,\bn)=s_j(t,\bm,\bn)s_i(t-n_i-n_k,\bm+\be_j,\bn).
\end{equation}
Hence, the functions~$F_i$ and~$F_j$ satisfy the functional equation
\begin{equation}
 \label{eq_Z1Z2}
F_i(t,\bm,\bn)+F_j(t-n_j-n_k,\bm+\be_i,\bn)=F_j(t,\bm,\bn)+F_i(t-n_i-n_k,\bm+\be_j,\bn).
\end{equation}

\end{propos}

\begin{propos}
 The functions $F_i$ satisfy the functional equations
 \begin{multline}\label{eq_Z6term}
  F_i(t+1,\bm,\bn)+F_i(t,  \bm,\bn+\be_j)+F_i(t-1,\bm,\bn-\be_j)\\
  {}=F_i(t-1,\bm,\bn)+F_i(t,\bm,\bn-\be_j)+F_i(t+1,\bm,\bn+\be_j),
  \end{multline}
 where $i,j,k$ is a permutation of~$1,2,3$.
\end{propos}
\begin{proof}
Consider the following difference operators on functions on~$\Z^7$:
\begin{align*}
 D_1&=\Delta_{(0,\bnul,\be_j)},&
 D_2&=\Delta_{(-1,\bnul,-\be_j)},&
 D_3&=\Delta_{(1,\bnul,\be_i)}.
\end{align*}
By Proposition~\ref{propos_independ} we know that $F_i(t,\bm,\bn)$ does not depend on~$n_i$ and $F_j(t,\bm,\bn)$ does not depend on~$n_j$. It follows that
\begin{align*}
 D_1\bigl(F_j(t,\bm,\bn)\bigr)&=0,\\
 D_2\bigl(F_j(t-n_j-n_k,\bm+\be_i,\bn)\bigr)&=0,\\
 D_3\bigl(F_i(t-n_i-n_k,\bm+\be_j,\bn)\bigr)&=0.
\end{align*}
Moreover, the operators~$D_1$, $D_2$, and~$D_3$ pairwise commute, since they are difference operators with constant coefficients. Therefore, applying the operator~$D_1D_2D_3$ to both sides of equation~\eqref{eq_Z1Z2}, we obtain that
$$
D_1D_2D_3\bigl(F_i(t,\bm,\bn)\bigr)=0,
$$
that is,
\begin{multline*}
 F_i(t,\bm,\bn+\be_i)-F_i(t-1,\bm,\bn)-F_i(t,\bm,\bn+\be_i-\be_j)-F_i(t+1,\bm,\bn+\be_i+\be_j) \\
 {}+F_i(t,\bm,\bn+\be_j)+F_i(t-1,\bm,\bn-\be_j)+F_i(t+1,\bm,\bn+\be_i)-F_i(t,\bm,\bn)=0.
\end{multline*}
Since $F_i(t,\bm,\bn)$ does not depend on~$n_i$, we obtain equation~\eqref{eq_Z6term}.
\end{proof}

We will also need the following proposition, which is a direct consequence of formula~\eqref{eq_z}.

\begin{propos}
 We have $s_i(t,\bm,\bnul)=1$ for all $t$ and~$\bm$. Moreover, $s_1(t,\bnul,\be_3)=1$ for all~$t$. Hence,
\begin{align}
 \label{eq_Z=0-1}
 F_i(t,\bm,\bnul)&=0,\\
  \label{eq_Z=0-1spec}
 F_1(t,\bnul,\be_3)&=0.
\end{align}
\end{propos}

\begin{remark}
 The motivation for considering the particular mapping classes~$s_i(t,\bm,\bn)$ given by~\eqref{eq_z} stems from the Reidemeister rewriting process, see~\cite[Section~2.3]{MKS}. Indeed, the elements~$c^t\bz^{\bm}\bfy^{\bn}$ form a system of representatives of some (not all) cosets of the stabilizer $\mathcal{S}=\Stab_{\K_g^b}(\alpha_1\cup\alpha_1'\cup\alpha_2\cup\alpha_2'\cup\alpha_3\cup\alpha_3')$ in the (componentwise) stabilizer of the same multicurve in the whole mapping class group~$\Mod(\Sigma_g^b)$. Then the elements~$s_i(t,\bm,\bn)$ become some of the standard generators for~$\mathcal{S}$ used in the Reidemeister rewriting process and~\eqref{eq_sisj} is the result of applying the rewriting process to the relation~$z_iz_j=z_jz_i$.
\end{remark}

\section{Quasipolynomial property for~$F_i$}\label{section_quasi}

The aim of this section is to prove the following proposition; its proof will rely solely on equations~\eqref{eq_Z1Z2}, \eqref{eq_Z6term}, \eqref{eq_Z=0-1}, and~\eqref{eq_Z=0-1spec}. Recall that a definition of a quasipolynomial was given in Section~\ref{section_equations}.

\begin{propos}\label{propos_Z_quasi}
 The functions $F_i(t,\bm,\bn)$ are quasipolynomials in~$t$ of degree~$\le 3$ for any fixed~$\bm$ and~$\bn$.
\end{propos}

Before proving this proposition, we first derive its corollary.

\begin{cor}\label{cor_quasi_Fu}
 For any $i$ and any~$(t,\bm,\bn)\in\Z^7$, we have that
 $$
 (e_{u_L}+1)(e_{u_L}-1)^4\cdot F_i(t,\bm,\bn)=0.
 $$
\end{cor}

\begin{proof}
 By Proposition~\ref{propos_tau_zy} we have $\tau(c)=u_L$. Hence,
 $$
 e_{u_L}\cdot F_i(t,\bm,\bn)=c\cdot F_i(t,\bm,\bn)=F_i(t+1,\bm,\bn).
 $$
 Thus, the corollary follows immediately from Proposition~\ref{propos_Z_quasi} and the definition of a quasipolynomial.
\end{proof}

In the remainder of this section we prove Proposition~\ref{propos_Z_quasi}.

\begin{propos}\label{propos_phipsichitheta}
 Suppose that $i,j,k$ is a cyclic permutation of~$1,2,3$. Then
 \begin{multline}\label{eq_phipsichitheta}
  F_i(t,\bm,\bn)=\varphi_i(t,\bm)+\psi_i(t-n_j,\bm)+\chi_i(t-n_k,\bm)\\{}+\theta_i(t-n_j-n_k,\bm) + \nu_i(\bm,n_j,n_k)
 \end{multline}
 for some $(\K_g^b)^{\ab}$-valued functions~$\varphi_i$, $\psi_i$, $\chi_i$, $\theta_i$, and~$\nu_i$.
\end{propos}

\begin{proof}
 Let us prove the proposition for $i=1$. In this proof we will never use the only unsymmetrical equation~\eqref{eq_Z=0-1spec}, so the proofs for~$i=2$ and~$i=3$ can be obtained by cyclic permutations of indices~$1,2,3$.

By Proposition~\ref{propos_independ} the value~$F_1(t,\bm,\bn)$ does not depend on~$n_1$.
 Consider equations~\eqref{eq_Z6term} for the triples $(i,j,k)=(1,2,3)$ and $(i,j,k)=(1,3,2)$. They are functional equations of the form~\eqref{eq_func1}. Hence, by Proposition~\ref{propos_func1} we obtain that
  \begin{align}\label{eq_abg_tilde}
   F_1(t,\bm,\bn)&=\kappa(t,\bm,n_3)+\mu(t-n_2,\bm,n_3)+\lambda(\bm,n_2,n_3)\\
               {}&=\tkappa(t,\bm,n_2)+\tmu(t-n_3,\bm,n_2)+\tlambda(\bm,n_2,n_3)\nonumber
  \end{align}
  for some functions~$\kappa$, $\mu$, $\lambda$, $\tkappa$, $\tmu$, and~$\tlambda$.
 Applying the partial difference operator $\Delta_t$, we get
\begin{equation*}
  \kappa'(t,\bm,n_3)+\mu'(t-n_2,\bm,n_3)
               =\tkappa'(t,\bm,n_2)+\tmu'(t-n_3,\bm,n_2).
 \end{equation*}
 This is a functional equation of the form~\eqref{eq_func2}.
By Proposition~\ref{propos_func2} we obtain that
 \begin{align*}
  \kappa'(t,\bm,n_3)&=g_1(t,\bm)+g_2(t-n_3,\bm)+h(\bm,n_3),\\
  \mu'(t,\bm,n_3)&=g_3(t,\bm)+g_4(t-n_3,\bm)-h(\bm,n_3)
 \end{align*}
 for some functions~$g_i$ and~$h$. Let $\varphi$, $\psi$, $\chi$, and~$\theta$ be the solutions of the equations
 \begin{align*}
  \varphi'(t,\bm)&=g_1(t,\bm),&\psi'(t,\bm)&=g_3(t,\bm),\\
  \chi'(t,\bm)&=g_2(t,\bm),&\theta'(t,\bm)&=g_4(t,\bm)
 \end{align*}
 with the initial conditions
 $$
 \varphi(0,\bm)=\psi(0,\bm)=\chi(0,\bm)=\theta(0,\bm)=0.
 $$
Such solutions exist, are unique and are given by the formulas
$$
\varphi(t,\bm)=\left\{
\begin{aligned}
&\sum_{s=0}^{t-1}g_1(s,\bm)&&\text{if}\ t\ge 0,\\
-&\sum_{s=t}^{-1}g_1(s,\bm)&&\text{if}\ t<0,
\end{aligned}
\right.
$$
and similar formulas for~$\psi$, $\chi$, and~$\theta$. Then
\begin{align*}
  \Delta_t\bigl(\kappa(t,\bm,n_3)-\varphi(t,\bm)-\chi(t-n_3,\bm)\bigr)&=h(\bm,n_3),\\
  \Delta_t\bigl(\mu(t,\bm,n_3)-\psi(t,\bm)-\theta(t-n_3,\bm)\bigr)&=-h(\bm,n_3).
 \end{align*}
 Therefore,
\begin{align*}
  \kappa(t,\bm,n_3)&=\varphi(t,\bm)+\chi(t-n_3,\bm)+th(\bm,n_3)+p(\bm,n_3),\\
  \mu(t,\bm,n_3)&=\psi(t,\bm)+\theta(t-n_3,\bm)-th(\bm,n_3)+q(\bm,n_3)
 \end{align*}
 for some functions~$p$ and~$q$. Substituting this to~\eqref{eq_abg_tilde}, adding index~$1$ to the functions~$\varphi$, $\psi$, $\chi$, and~$\theta$ and taking
 $$
 \nu_1(\bm,n_2,n_3)=\lambda(\bm,n_2,n_3)+n_2h(\bm,n_3)+p(\bm,n_3)+q(\bm,n_3),
 $$
 we obtain the required formula~\eqref{eq_phipsichitheta}.
\end{proof}

 \begin{propos}\label{propos_phipsichitheta_quasi}
 Let $\varphi_i$, $\psi_i$, $\chi_i$, and~$\theta_i$ be the functions from Proposition~\ref{propos_phipsichitheta}. Then the following statements hold.
 \begin{enumerate}
 \item The functions~$\psi_i(t,\bm)$ and~$\chi_i(t,\bm)$ are quasipolynomials in~$t$ of degree~$\le3$.

 \item There exists an $(\K_g^b)^{\ab}$-valued function~$\omega$ on~$\Z$ such that
 \begin{enumerate}
  \item  $\omega$ is a quasipolynomial of degree~$\le 4$,
 \item for each~$i$ and each $\bm$, the functions
 $\varphi_i(t,\bm)-\omega(t)$ and $\theta_i(t,\bm)+\omega(t)$
 are quasipolynomials in~$t$ of degree~$\le 3$.
  \end{enumerate}
   \end{enumerate}
 \end{propos}

\begin{proof}
 Suppose that $i,j,k$ is a cyclic permutation of~$1,2,3$.
 We substitute expressions~\eqref{eq_phipsichitheta} into~\eqref{eq_Z1Z2}, apply the partial difference operator~$\Delta_t$ to both sides, and then redenote~$t-n_k$ by~$t$. Then we obtain the following equation:
 \begin{equation*}
 \begin{split}
  \bigl(\varphi_i'&(t+n_k,\bm)-\varphi_j'(t+n_k,\bm)\bigr)-\chi_j'(t-n_i+n_k,\bm)+\psi_i'(t-n_j+n_k,\bm)\\
  {}&+\bigl(\chi_i'(t,\bm)-\psi_j'(t,\bm)\bigr)
  -\bigl(\varphi_i'(t-n_i,\bm+\be_j)+\theta_j'(t-n_i,\bm)\bigr)\\
  {}&+\bigl(\varphi_j'(t-n_j,\bm+\be_i)+\theta_i'(t-n_j,\bm)\bigr)\\
  {}&+\bigl(\chi_j'(t-n_i-n_j,\bm+\be_i)-\psi_i'(t-n_i-n_j,\bm+\be_j)\bigr)\\
  {}&-\chi_i'(t-n_i-n_k,\bm+\be_j)+\psi_j'(t-n_j-n_k,\bm+\be_i)\\
  {}&+\bigl(\theta_j'(t-n_i-n_j-n_k,\bm+\be_i)-\theta_i'(t-n_i-n_j-n_k,\bm+\be_j)\bigr)=0.
 \end{split}
 \end{equation*}
This is an equation of the form~\eqref{eq_9}. Therefore, by Proposition~\ref{propos_9} the following functions are quasipolynomials in~$t$ of degree~$\le 2$:
\begin{gather*}
 \psi_i'(t,\bm),\qquad\chi_i'(t,\bm),\qquad \varphi_i'(t,\bm)-\varphi_j'(t,\bm),\\ \varphi_i'(t,\bm+\be_j)+\theta_j'(t,\bm).
\end{gather*}
Hence, $\psi_i(\bm,t)$ and~$\chi_i(\bm,t)$ are quasipolynomials in~$t$ of degree~$\le3$. Moreover,
\begin{equation}\label{eq_phi'_mod}
 \varphi_i'(t,\bm)\equiv \varphi_1'(t,\bm)\pmod{\CQ_2}.
\end{equation}

Next, we substitute expressions~\eqref{eq_phipsichitheta} into~\eqref{eq_Z=0-1} and apply the operator~$\Delta_t$. We get
\begin{equation*}
  \varphi'_i(t,\bm)+\psi'_i(t,\bm)+\chi'_i(t,\bm)+\theta'_i(t,\bm)=0.
 \end{equation*}
Hence,
\begin{equation}\label{eq_theta'_mod}
 \theta_i'(t,\bm)\equiv-\varphi_i'(t,\bm)\equiv -\varphi_1'(t,\bm)\pmod{\CQ_2}.
\end{equation}
Now, the condition that $\varphi_i'(t,\bm+\be_j)+\theta_j'(t,\bm)$  is a quasipolynomial in~$t$ of degree~$\le2$ reads as
$$
\varphi_1'(t,\bm+\be_j)\equiv \varphi_1'(t,\bm)\pmod{\CQ_2}.
$$
Since this equivalence holds for all~$j$, we obtain that
\begin{equation}\label{eq_phi1'_mod}
\varphi_1'(t,\bm)\equiv \varphi_1'(t,\bnul)\pmod{\CQ_2}
\end{equation}
for all $\bm\in\Z^3$.

We set
$$
\omega(t)=\varphi_1(t,\bnul).
$$
From~\eqref{eq_phi'_mod}, \eqref{eq_theta'_mod}, and~\eqref{eq_phi1'_mod} it follows that the functions $\varphi_i(t,\bm)-\omega(t)$ and $\theta_i(t,\bm)+\omega(t)$ are quasipolynomials in~$t$ of degree~$\le3$. So we suffice to prove that $\omega$ is a quasipolynomial of degree~$\le4$. To do this, we substitute expression~\eqref{eq_phipsichitheta} into relation~\eqref{eq_Z=0-1spec}. We obtain that
$$
\varphi_1(t,\bnul)+\psi_1(t,\bnul)+\chi_1(t-1,\bnul)+\theta_1(t-1,\bnul)+\nu_1(\bnul,0,1)=0.
$$
Therefore,
$
\omega(t)-\omega(t-1)
$
is a quasipolynomial of degree~$\le3$. Thus, $\omega$ is a quasipolynomial of degree~$\le4$.
\end{proof}

\begin{proof}[Proof of Proposition~\ref{propos_Z_quasi}]
 By Propositions~\ref{propos_phipsichitheta} and~\ref{propos_phipsichitheta_quasi} we have that
\begin{equation}\label{eq_Z_omega}
F_i(t,\bm,\bn)\equiv \omega(t)-\omega(t-n_j-n_k)\pmod{\CQ_3},
\end{equation}
where $\omega$ is a quasipolynomial of degree~$\le4$ and $i,j,k$ is a permutation of~$1,2,3$. Then the right-hand side of~\eqref{eq_Z_omega} is a quasipolynomial in~$t$ of degree~$\le3$. Therefore, $F_i(t,\bm,\bn)$  is also a quasipolynomial in~$t$ of degree~$\le3$.
\end{proof}

\section{Proof of Theorem~\ref{thm_explicit}}\label{section_proof_explicit}

\begin{propos}
 We have the following relation in~$\K_g^b$:
 \begin{equation}\label{eq_h}
 x_2\,s_2(1,\bnul,\be_1)^{-1}x_2^{-1}c\,T_{\gamma}^{-1}c^{-1}s_2(1,-\be_3,\be_1)T_{\gamma}=1.
 \end{equation}
\end{propos}

\begin{proof}
 By~\eqref{eq_z} we have
 \begin{align*}
  s_2(1,\bnul,\be_1)&=cy_1z_2y_1^{-1}z_2^{-1},\\
  s_2(1,-\be_3,\be_1)&=cz_3^{-1}y_1z_2y_1^{-1}z_2^{-1}z_3.
 \end{align*}
 Substitute these expressions and also expression~\eqref{eq_delta} for~$T_{\gamma}$ into~\eqref{eq_h}. Using that
 \begin{itemize}
  \item $x_2$ commutes with~$c$,
  \item each of the elements~$y_1$ and~$z_2$ commutes with each of the elements~$x_3$ and~$y_2$,
  \item the twists~$T_{\alpha_i}$ and~$T_{\alpha_i'}$ commute with all mapping classes under consideration,
 \end{itemize}
 we see that the left-hand side of~\eqref{eq_h} reads as
 \begin{multline*}
 x_2(z_2y_1z_2^{-1}y_1^{-1}c^{-1})x_2^{-1}c(x_2x_3y_2z_3)c^{-1}(cz_3^{-1}y_1z_2y_1^{-1}z_2^{-1}z_3)(z_3^{-1}y_2^{-1}x_3^{-1}x_2^{-1})\\
 {}=x_2(z_2y_1z_2^{-1}y_1^{-1})x_3y_2(y_1z_2y_1^{-1}z_2^{-1})y_2^{-1}x_3^{-1}x_2^{-1}=1,
 \end{multline*}
 as desired.
\end{proof}

\begin{proof}[Proof of Theorem~\ref{thm_explicit}]
Since $\tau(c)=u_L$, relation~\eqref{eq_h} yields the
following relation in~$(\K_{g}^b)^{\ab}$:
\begin{equation}\label{eq_h_ab}
 (e_{u_L}-1)\cdot [T_{\gamma}]= -x_2\cdot F_2(1,\bnul,\be_1)+F_2(1,-\be_3,\be_1).
\end{equation}
Apply to both sides of this relation the element
$$
q = (e_{u_L}+1)(e_{u_L}-1)^4
$$
of the group ring~$\Z U^b_g$. By Corollary~\ref{cor_quasi_Fu} we know that $q$ annihilates all elements~$F_i(t,\bm,\bn)$. Since $c$ commutes with $x_2$,  the actions of the elements~$q$ and~$x_2$ on~$(\K_{g}^b)^{\ab}$ also commute with each other. Therefore, $q$ annihilates the right-hand side of~\eqref{eq_h_ab}. The theorem follows.
\end{proof}

\begin{remark}
 The above proof yields that the relation
 $$
 (e_{u_L}+1)(e_{u_L}-1)^5\cdot [T_{\gamma}]=0
 $$
 holds not only in~$(\K_g^b)^{\ab}$ but also in~$\mathcal{S}^{\ab}$, where $\mathcal{S}$ is the stabilizer of the multicurve $\alpha_1\cup\alpha_1'\cup\alpha_2\cup\alpha_2'\cup\alpha_3\cup\alpha_3'$ in~$\K_g$.
\end{remark}

\section{Subgroups of~$U_g$ associated with a separating simple closed curve}\label{section_subgroups}

Consider an essential (i.\,e., not bounding a disk) separating simple closed curve~$\gamma$ on a closed surface~$\Sigma_g$ of genus~$g\ge3$. Let $S_1$ and~$S_2$ be the connected components of~$\Sigma_g\setminus \gamma$. Suppose that the genera of~$S_1$ and~$S_2$ are $g_1$ and~$g_2$, respectively. Then $g_1+g_2=g$. Define the \textit{genus} of~$\gamma$ to be the smallest of the two numbers~$g_1$ and~$g_2$. We set $P_i =H_1(S_i,\Z)$. Then $H=P_1\oplus P_2$. Choose a symplectic basis~$e_1,\ldots,e_{2g_1}$ of~$P_1$ and a symplectic basis~$f_1,\ldots,f_{2g_2}$ of~$P_2$ so that $e_{2i-1}\cdot e_{2i}=1$, $f_{2i-1}\cdot f_{2i}=1$, and all other intersection numbers are equal to zero.

We denote by~$\I_{\gamma}$ the stabilizer of~$\gamma$ in the Torelli group~$\I_g$ and set $U_{\gamma}=\tau(\I_{\gamma})$.

\begin{propos}
The group~$U_{\gamma}$ has a basis consisting of~$\binom{2g_1}{3}$ elements $e_i\wedge e_j\wedge e_k$, where $1\le i<j <k\le 2g_1$, and $\binom{2g_2}{3}$ elements $f_i\wedge f_j\wedge f_k$, where $1\le i<j <k\le 2g_2$.
\end{propos}

\begin{proof}Denote by~$\oS_1$ and~$\oS_2$ the closures of~$S_1$ and~$S_2$, respectively.
 Every element $h\in\I_{\gamma}$ can be written as a product $h_1h_2$, where $h_1$ is supported on~$\oS_1$ and $h_2$ is supported on~$\oS_2$. Then~$h_1$ and~$h_2$ lie in the Torelli groups~$\I(\oS_1)$ and~$\I(\oS_2)$ of surfaces~$\oS_1$ and~$\oS_2$, respectively. Vice versa, any element of~$\I(\oS_i)$, where $i\in\{1,2\}$,  being extended by the identity becomes an element of~$\I_{\gamma}$. The images of the Torelli groups~$\I(\oS_1)$ and~$\I(\oS_2)$ under the Johnson homomorphism are exactly the groups~$\exter^3 P_1$ and~$\exter^3P_2$, respectively. Finally, one can check directly that the composition
 $$
 \exter^3P_1\oplus\exter^3P_2\hookrightarrow \exter^3H\twoheadrightarrow (\exter^3H)/(\Omega\wedge H) =U_g
 $$
 is injective. The proposition follows.
\end{proof}

We denote by~$W_{\gamma}^{(1)}$ the subgroup of~$U_g$ generated by all elements $a\wedge b\wedge c$ such that $a\in P_1$, $b,c\in P_2$, and $b\cdot c=0$. The subgroup $W_{\gamma}^{(2)}\subset U_g$ is defined likewise by swapping the roles of~$P_1$ and~$P_2$. Obviously, $W_{\gamma}^{(1)}\cap W_{\gamma}^{(2)}=\{0\}$. We set $W_{\gamma}=W_{\gamma}^{(1)}\oplus W_{\gamma}^{(2)}$.
The following proposition is straightforward.

\begin{propos}\label{propos_W_basis}
Set $r_1=2g_1(2g_2^2-g_2-1)$ and
let $w_1,\ldots,w_{r_1}$ be the following $r_1$ elements of~$U_g$ listed in some order:
\begin{itemize}
 \item the $4g_1g_2(g_2-1)$ elements $e_i\wedge f_j\wedge f_k$, where $1\le i\le 2g_1$, $1\le j<k\le 2g_2$, and $(j,k)\ne (2l-1,2l)$,
 \item the $2g_1(g_2-1)$ elements $e_i\wedge (f_{2l-1}+f_1)\wedge (f_{2l}-f_2)$, where $1\le i\le 2g_1$ and $2\le l\le g_2$.
\end{itemize}
Then $w_1,\ldots,w_{r_1}$ is a basis of~$W_{\gamma}^{(1)}$.
\end{propos}

Similarly, we set $r_2=2g_2(2g_1^2-g_1-1)$ and denote by $w_{r_1+1},\ldots,w_{r_1+r_2}$ an analogous basis of~$W_{\gamma}^{(2)}$ obtained by swapping the roles of~$P_1$ and~$P_2$ in Proposition~\ref{propos_W_basis}. We set $$r=r_1+r_2=4g_1g_2(g-1)-2g.$$
Then $w_1,\ldots,w_r$ is a basis of~$W_{\gamma}$.

It is easy to see that each~$w_i$ is an element of the form~$u_L$, where $L$ is a subgroup that satisfy all conditions from Theorem~\ref{thm_explicit}. So by Theorem~\ref{thm_explicit} we have
$$
(e_{w_i}+1)(e_{w_i}-1)^5\cdot [T_{\gamma}]=0,\qquad i=1,\ldots,r.
$$
We set $V_{\gamma}=2W_{\gamma}$, $V_{\gamma}^{(s)}=2W_{\gamma}^{(s)}$ for $s=1,2$, and $v_i=2w_i$ for $i=1,\ldots,r$. Then $v_1,\ldots,v_r$ is a basis of~$V_{\gamma}$. Since $e_{2u} = e_u^2$, it follows that
\begin{equation}\label{eq_v5}
 (e_{v_i}-1)^5\cdot [T_{\gamma}]=0,\qquad i=1,\ldots,r.
\end{equation}

\begin{propos}\label{propos_V_contain}
We have $U_{\gamma}\cap V_{\gamma}=\{0\}$, so the sum of~$U_{\gamma}$ and~$V_{\gamma}$ in~$U_g$ is direct. Let $m$ be the least common multiple of~$g_1$ and~$g_2$. Then the group~$2mU_g$ is contained in~$U_{\gamma}\oplus V_{\gamma}$. Moreover, if $g=3$, then the group~$\tV_3$ described in Subsection~\ref{subs_nilpotence} is also contained in~$U_{\gamma}\oplus V_{\gamma}$.
\end{propos}

\begin{proof}
 Let us prove that the group~$2mU_g$ is contained in the sum $U_{\gamma}+V_{\gamma}$. This will immediately imply that the sum $U_{\gamma}+V_{\gamma}$ is direct, since
 \begin{equation}\label{eq_ranks}
 \rk(2mU_g)=\binom{2g}{3}-2g=\binom{2g_1}{3}+\binom{2g_2}{3}+r=\rk (U_{\gamma})+\rk(V_{\gamma}).
 \end{equation}
 We need to prove that every monomial in~$e_i$'s and~$f_j$'s being multiplied by~$2m$ gets into $U_{\gamma}+V_{\gamma}$. This is obvious for the monomials $e_i\wedge e_j\wedge e_k$ and $f_i\wedge f_j\wedge f_k$, since they lie in~$U_{\gamma}$, and for the monomials $e_i\wedge f_j\wedge f_k$ and $f_i\wedge e_j\wedge e_k$ with $(j,k)\ne (2l-1,2l)$, since they lie in~$W_{\gamma}$. So it remains to consider a monomial $e_i\wedge f_{2l-1}\wedge f_{2l}$. (The case of a monomial $f_i\wedge e_{2l-1}\wedge e_{2l}$ is similar.) We have the following equalities in~$U_g$:
 \begin{align*}
 2g_2\,e_i\wedge f_{2l-1}\wedge f_{2l}&= 2g_2\,e_i\wedge f_{2l-1}\wedge f_{2l}-2e_i\wedge\left(\sum_{j=1}^{g_1}e_{2j-1}\wedge e_{2j}+\sum_{k=1}^{g_2}f_{2k-1}\wedge f_{2k} \right)\\
 &=-2\sum_{j=1}^{g_1}e_i\wedge e_{2j-1}\wedge e_{2j}+2g_2\,e_i\wedge(f_{2l-1}\wedge f_{2l}-f_{1}\wedge f_2)\\
 &\ \ \,-
 2\sum_{k=2}^{g_2}e_i\wedge(f_{2k-1}\wedge f_{2k}-f_1\wedge f_2).
 \end{align*}
The monomials~$e_i\wedge e_{2j-1}\wedge e_{2j}$ lie in~$U_{\gamma}$. Moreover, for each $k\ne 1$, we have that
\begin{multline*}
 2e_i\wedge\left(f_{2k-1}\wedge f_{2k}-f_1\wedge f_2 \right)=
 2e_i\wedge (f_{2k-1}+f_1)\wedge (f_{2k}-f_2)\\-2e_i \wedge f_1\wedge f_{2k}-2e_i\wedge f_2\wedge f_{2k-1},
\end{multline*}
 which is a linear combination of basis elements of~$V_{\gamma}$. Hence, the element $2g_2e_i\wedge f_{2l-1}\wedge f_{2l}$, and therefore the element $2me_i\wedge f_{2l-1}\wedge f_{2l}$, belongs to~$U_{\gamma}+V_{\gamma}$. This completes the proof of the inclusion $2mU_g\subset U_{\gamma}\oplus V_{\gamma}$.

 Now, suppose that $g=3$ and prove that $\tV_3\subset U_{\gamma}\oplus V_{\gamma}$. We may assume that $g_1=1$ and~$g_2=2$. One can verify directly that the following $14$ elements form a basis of~$\tV_3$:
 \begin{itemize}
  \item $4f_i\wedge f_j\wedge f_k$ for $1\le i<j<k\le 4$,
  \item $4e_i\wedge f_1\wedge f_2$ for $i\in\{1,2\}$,
  \item $2e_i\wedge f_j\wedge f_k$ for $i\in\{1,2\}$, $1\le  j<k\le 4$, $(j,k)\ne (1,2)$, $(j,k)\ne (3,4)$.
  \end{itemize}
 The above argument shows that all these elements belong to~$U_{\gamma}\oplus V_{\gamma}$.
 \end{proof}

 \begin{remark}\label{remark_not_Kg1}
  Proposition~\ref{propos_V_contain} is precisely the step where the proof of Theorem~\ref{thm_nilpotent} breaks in the case of the group~$\K_g^1$. Indeed, though we can similarly define the subgroups~$U_{\gamma}$ and~$V_{\gamma}$, equality~\eqref{eq_ranks} does not hold any more, so $U_{\gamma}\oplus V_{\gamma}$ turns out to be a subgroup of infinite index in~$U_g^1$.
 \end{remark}

 Recall that in Subsection~\ref{subs_nilpotence} we have defined~$V_g$ to be the group $2(g-1)(g-2)U_g$.

\begin{cor}\label{cor_V_contain}
 If the genus of~$\gamma$ is either\/~$1$ or\/~$2$, then $U_{\gamma}\oplus V_{\gamma}$ contains~$V_g$.
\end{cor}

\begin{propos}\label{propos_Y_nilpotent}
The $\Z (U_{\gamma}\oplus V_{\gamma})$-submodule of~$\K_g^{\ab}$ generated by~$[T_{\gamma}]$ is nilpotent  of nilpotency index no greater than~$4r+1$.
\end{propos}

\begin{proof}
 Since $U_{\gamma}=\tau(\I_{\gamma})$, we see that
 $
 (e_u-1)\cdot [T_{\gamma}]=0
 $
 for all $u\in U_{\gamma}$. Also we have~\eqref{eq_v5} for the basis $v_1,\ldots,v_r$ of~$V_{\gamma}$. The proposition follows.
\end{proof}

We have $r=4g^2-10g+4$ when the genus of~$\gamma$ is~$1$ and $r=8g^2-26g+16$ when the genus of~$\gamma$ is~$2$. For $g\ge4$, the latter value of~$r$ is always greater than the former. So the following statement is a direct consequence of Corollary~\ref{cor_V_contain} and Propositions~\ref{propos_V_contain} and~\ref{propos_Y_nilpotent}.

\begin{cor}\label{cor_V_nilpotent}
 If the genus of~$\gamma$ is either~$1$ or~$2$, then the $\Z V_g$-submodule of~$\K_g^{\ab}$ generated by~$[T_{\gamma}]$ is nilpotent  of nilpotency index no greater than $32g^2-104g+65$. If $g=3$, then the $\Z \tV_3$-submodule of~$\K_3^{\ab}$ generated by~$[T_{\gamma}]$ is nilpotent of nilpotency index no greater than~$41$.
\end{cor}

\section{Proof of Theorem~\ref{thm_nilpotent} and quantitative results}\label{section_proof_nilpotent}

\begin{proof}[Proof of Theorem~\ref{thm_nilpotent}]
 By a result of Johnson~\cite{Joh85a}, \cite{Joh85b} the group~$\K_g$ is generated by twists~$T_{\gamma}$ about genus~$1$ and genus~$2$ separating simple closed curves\footnote{Since Johnson was not careful enough in writing this result, let us explain how to extract it from the papers~\cite{Joh85a} and~\cite{Joh85b}. This statement is formulated at the end of~\S 1 of~\cite{Joh85b}. Nevertheless, its proof is (implicitly) contained not in~\cite{Joh85b} but in the previous paper~\cite{Joh85a}. Namely, the main result of~\cite{Joh85a} is that the subgroup~$\CT_g$ of~$\Mod(S_g)$ generated by all Dehn twists about separating simple closed curves coincides with the group $\K_g=\ker\tau$. Johnson's proof remains literally the same if one originally defines~$\CT_g$ to be the subgroup generated only by twists about separating simple closed curves of  genera~$1$ and~$2$.}. Hence, the group~$\K_g^{\ab}$ is generated by the corresponding elements~$[T_{\gamma}]$. By Corollary~\ref{cor_V_nilpotent} each such element generates a nilpotent $\Z V_g$-module with the uniform bound~$32g^2-104g+65$ on the nilpotency index. Thus, the whole $\Z V_g$-module  $\K_g^{\ab}$ is also nilpotent, with the same bound on the nilpotency index. For $g=3$, the same remains true if we replace~$V_3$ with~$\tV_3$.
\end{proof}

The above argument in fact yields the following explicit, albeit enormous, bound on the number of generators of the group~$\K_3^{\ab}$.

\begin{theorem}\label{thm_quant}
 The group $\K_3^{\ab}$ can be generated by
 $$
\binom{54}{14}\cdot \left(17\cdot 2^{21}+92\right)\approx 1.2\cdot 10^{20}
 $$
 elements. Moreover,
 $$
 \dim H_1(\K_3,\Q)\le \binom{54}{14}\cdot 105\approx 3.4\cdot 10^{14}.
 $$
\end{theorem}

\begin{remark}
Similar bounds can be written for all~$\K_g$ with $g\ge4$. Nevertheless, they are not as interesting, since they are even more enormous. Moreover, for large~$g$, the bound on the number of generators of~$\K_g^{\ab}$ deduced from Theorem~\ref{thm_nilpotent} is asymptotically worse than the bound on the number of generators of~$\K_g$ obtained by Ershov and Franz~\cite{ErFr23}.
\end{remark}

To prove Theorem~\ref{thm_quant} we first need the following standard lemma. For a $\Z A$-module~$M$, we denote by~$M^A$ and~$M_A$ the groups of invariants and coinvariants of~$M$, respectively.

\begin{lem}\label{lem_alg}
 Let $A$ be a free abelian group of finite rank~$r$ and $M$ a nilpotent $\Z A$-module of nilpotency index~$k$. Assume that $M_A$ is a finitely generated abelian group with $n$ generators. Then $M$ is a finitely generated abelian group with no more than $\binom{k+r-1}{r}n$ generators.
\end{lem}

\begin{proof}
 We prove the lemma by the induction on~$k$.

 \textsl{Base case.} Suppose that $k=1$; then $M=M_A$ is an abelian group generated by~$n$ elements.

 \textsl{Induction step.} Let $v_1,\ldots,v_r$ be a basis of~$A$. For $x\in M$, consider the $M$-valued function on~$A$ given by $a\mapsto e_a\cdot x$.
 Now, for a vector $\bk=(k_1,\ldots,k_r)\in\Z_{\ge 0}^r$ with $|\bk|=k_1+\cdots+k_r=k-1$, we set
 \begin{equation}\label{eq_fbk}
 f_{\bk}(x)=\Delta_{v_1}^{k_1}\cdots \Delta_{v_r}^{k_r}(e_a\cdot x).
 \end{equation}
Here the right-hand side does not dependend of~$a$ because of the nilpotency of~$M$. It follows immediately from~\eqref{eq_fbk} that $f_{\bk}(e_a\cdot x)=e_a\cdot f_{\bk}(x)=f_{\bk}(x)$ for all~$a$. So we obtain a well-defined homomorphism of $\Z A$-modules $f_{\bk}\colon M\to M^A$, which certainly factors through the group~$M_A$:
$$
f_{\bk}\colon M\xrightarrow{\text{projection}}M_A\longrightarrow M^A.
$$
Hence the image $\im (f_{\bk})$ is a finitely generated abelian group with at most $n$ generators.

Now, let $N\subset M^A$ be the sum of the subgroups~$\im(f_{\bk})$ over all~$\bk$ such that $|\bk|=k-1$. Since the number of vectors~$\bk\in\Z_{\ge}^r$ with $|\bk|=k-1$ is equal to~$\binom{k+r-2}{r-1}$, we see that $N$ is a finitely generated abelian group with at most $\binom{k+r-2}{r-1}n$ generators. Consider the $\Z A$-module $M'=M/N$. It follows from~\eqref{eq_fbk} that $M'$ is nilpotent of nilpotency index~$\le k-1$. Moreover,
 the group of coinvariants $M'_A$ is a quotient group of~$M_A$, hence is generated by no more than $n$ elements.  Therefore, by the induction hypothesis, we have that $M'$ is a finitely generated abelian group with at most $\binom{k+r-2}{r}n$ generators. Since both abelian groups~$M'$ and~$N$ are finitely generated, we obtain that $M$ is finitely generated, too, with at most $\binom{k+r-1}{r}n$ generators.
\end{proof}

The proof of the following lemma repeats literally the proof of Lemma~\ref{lem_alg}.

\begin{lem}\label{lem_alg_Q}
 Let $A$ be a free abelian group of finite rank~$r$ and $M$ a nilpotent\/ $\Q A$-module of nilpotency index~$k$. Assume that $\dim_{\Q}M_A=n<\infty$. Then $\dim_{\Q}M\le \binom{k+r-1}{r}n$.
\end{lem}

\begin{proof}[Proof of Theorem~\ref{thm_quant}]
Let $R$ be either~$\Z$ or~$\Q$.

Consider the subgroup $\CG=\tau^{-1}\bigl(\tV_3\bigr)$ of~$\I_3$. Recall that $\I_3$ is generated by~$35$ elements, see~\cite{Joh83}. Since $|\I_3:\CG|=2^{20}$, the Schreier formula (see~\cite[Theorem~2.10]{MKS}) yields that $\CG$, and hence $H_1(\CG,\Z)$, can be generated by $34\cdot 2^{20}+1$ elements. Moreover, since $\CG$ is a subgroup of finite index of~$\K_3$, the result of Putman~\cite[Theorem~B]{Put18} yields that
$$
\dim H_1(\CG,\Q)=\dim(U_g\otimes\Q)=14.
$$
Consider the short exact sequence
$$
1\longrightarrow\K_3\longrightarrow\CG\stackrel{\tau}{\longrightarrow}\tV_3\longrightarrow 1.
$$
The corresponding $5$-term exact sequence of low-dimensional homology groups (see~\cite[Ch.~VII, Corollary~6.4]{Bro82}) reads as follows:
$$
H_2(\CG,R)\longrightarrow H_2\bigl(\tV_3,R\bigr)\longrightarrow H_1(\K_3,R)_{\tV_3} \longrightarrow H_1(\CG,R) \longrightarrow H_1\bigl(\tV_3,R\bigr)\longrightarrow 0.
$$
We have $H_2\bigl(\tV_3,R\bigr)\cong \exter^2R^{14}\cong R^{91}$. Hence, the abelian group $(\K_3^{\ab})_{\tV_3}=H_1(\K_3,\Z)_{\tV_3}$ can be generated by $17\cdot 2^{21}+92$ elements. Moreover, $\dim H_1(\K_3,\Q)_{\tV_3}\le 105$.  The statement of the theorem now follows from Theorem~\ref{thm_nilpotent} and Lemmas~\ref{lem_alg} and~\ref{lem_alg_Q}.
\end{proof}

\section{Concluding remarks and some open questions}\label{section_open}

The methods of the present paper allow only to study the abelianizations of the groups like~$\K_g^b$, $[\IA_n,\IA_n]$, and~$[\IO_n,\IO_n]$. It would be interesting to find out whether these methods can be adapted to prove the finite generation of those groups themselves. Note that the question of whether the groups~$\K_3$, $\K_3^1$, $[\IO_3,\IO_3]$, and~$[\IA_3,\IA_3]$ are finitely generated is still open. Also highly intriguing is the question of finite generation of the group~$[\IA_3,\IA_3]^{\ab}$, for which our method apparently does not work.
Of course the following  question is also open and interesting.

\begin{problem}
 Compute the abelianizations $(\K_g^b)^{\ab}$, $[\IO_n,\IO_n]^{\ab}$, and~$[\IA_n,\IA_n]^{\ab}$ or at least find a reasonable bounds on the numbers of their generators.
\end{problem}

The same (easier) question can be asked about the first homology groups with rational coefficients. In the case of~$\K_g$, this question is open only for $g=3,4,5$, since for $g\ge6$ the groups~$H_1(\K_g,\Q)$ were computed explicitly, see~\cite{DHP14},~\cite{MSS20}.
Note that Faes and Massuyeau~\cite{FaMa25} have recently proved the non-triviality of the torsion subgroup of~$(\K_g^b)^{\ab}$ for $g\ge 6$ and $b\in\{0,1\}$.

Also, it would be interesting to clarify the situation with the nilpotency condition.

\begin{quest}
 \textnormal{(1)} Are the groups~$(\K_g^b)^{\ab}$ for $3\le g\le 11$ nilpotent as $\Z U_g^b$-modules? If not, what is the largest subgroups $V\subset U_g^b$ for which $(\K_g^b)^{\ab}$ is a nilpotent $\Z V$-module?

 \textnormal{(2)} When the nilpotency condition holds, what is the exact value of (or at least a reasonable upper bound on) the nilpotency index?
\end{quest}

Similar questions can be asked  for the groups~$[\IO_n,\IO_n]^{\ab}$ and~$[\IA_n,\IA_n]^{\ab}$.
Note that, for $g\ge 6$, an explicit computation of the group~$H_1\bigl(\K_g^b,\Q\bigr)$ from~\cite{DHP14} and~\cite{MSS20} yields that it is a nilpotent $\Q U_g^b$-module of nilpotency index~$2$.

Finally, the following problem is motivated by Theorem~\ref{thm_explicit}. It can be considered as a step towards the computation of~$(\K_g^b)^{\ab}$.

\begin{problem}
 Suppose that $\gamma$ is a separating simple closed curve on~$\Sigma_g^b$. Describe the annihilator ideal for the class $[T_{\gamma}]\in (\K_g^b)^{\ab}$ in the Laurent polynomial ring~$\Z U_g^b$. In particular, let $u_L$ be as in Theorem~\ref{thm_explicit}; what is the minimal polynomial $P(u_L)$ that annihilates~$[T_{\gamma}]$?
\end{problem}

Theorem~\ref{thm_explicit} states that $P(u_L)$ divides the polynomial $(u_L+1)(u_L-1)^5$.

\end{document}